\documentclass[10pt]{amsart}
\usepackage{latexsym}
\usepackage{amsmath}
\usepackage{amssymb}
\usepackage{mathrsfs}
\usepackage{graphicx}
\usepackage{color}
\usepackage{pgfpages}
\usepackage{ifthen}
\usepackage{leftidx,tensor}
\usepackage[T1]{fontenc}
\usepackage[latin1]{inputenc}
\usepackage{mathtools}
\usepackage{comment}
\usepackage{dsfont}
\usepackage[nocompress]{cite}

\usepackage[shortlabels]{enumitem}
\usepackage{aliascnt}
\usepackage[bookmarks=true,pdfstartview=FitH, pdfborder={0 0 0}, colorlinks=true,citecolor=red, linkcolor=blue]{hyperref}
\usepackage{bbm}
\usepackage{nicefrac}

\theoremstyle{plain}
\newtheorem{thm}{Theorem}[section]
\newaliascnt{cor}{thm}
\newaliascnt{prop}{thm}
\newaliascnt{lem}{thm}

\newtheorem{prop}[prop]{Proposition}
\newtheorem{lem}[lem]{Lemma}
\aliascntresetthe{cor}
\aliascntresetthe{prop}
\aliascntresetthe{lem} 
\theoremstyle{definition}
\newaliascnt{defn}{thm}
\newaliascnt{asu}{thm}
\newaliascnt{con}{thm}

\newtheorem{asu}[asu]{Assumption}

\aliascntresetthe{defn}
\aliascntresetthe{asu}
\aliascntresetthe{con}
\newcounter{stp}
\newcounter{stpi}
\newcounter{stpci}
\newcounter{stpiii}

\theoremstyle{thm}
\newaliascnt{rem}{thm}
\newaliascnt{exa}{thm}
\newaliascnt{masu}{thm}
\newaliascnt{nota}{thm}
\newaliascnt{sett}{thm}
\newtheorem{rem}[rem]{Remark}

\aliascntresetthe{rem}
\aliascntresetthe{exa}
\aliascntresetthe{masu}
\aliascntresetthe{nota}
\aliascntresetthe{sett}
\numberwithin{equation}{section}

\setlist[enumerate]{font = \normalfont}

\newcommand{\rL}{\mathrm{L}}
\newcommand{\rLq}{\rL^q}
\newcommand{\rLp}{\rL^p}
\newcommand{\rD}{\mathrm{D}}
\newcommand{\rH}{\mathrm{H}}
\newcommand{\rW}{\mathrm{W}}
\newcommand{\rB}{\mathrm{B}}
\newcommand{\rC}{\mathrm{C}}
\newcommand{\rX}{\mathrm{X}}
\newcommand{\rY}{\mathrm{Y}}
\newcommand{\rF}{\mathrm{F}}
\newcommand{\rBUC}{\mathrm{BUC}}
\newcommand{\cL}{\mathcal{L}}
\newcommand{\cA}{\mathcal{A}}
\newcommand{\cR}{\mathcal{R}}
\newcommand{\N}{\mathbb{N}}
\newcommand{\R}{\mathbb{R}}
\newcommand{\E}{\mathbb{E}}
\newcommand{\F}{\mathbb{F}}
\DeclareMathOperator{\mdiv}{div}
\newcommand{\mre}{\mathrm{e}}
\newcommand{\del}{\partial}
\newcommand{\per}{\mathrm{per}}
\newcommand{\dz}{\del_z}
\newcommand{\dt}{\del_t}
\newcommand{\srd}{\, \mathrm{d}}
\newcommand{\eps}{\varepsilon}
\DeclareMathOperator{\Id}{Id}

\newcommand{\Gal}{\Gamma_l}
\newcommand{\Gau}{\Gamma_u}

\newcommand{\tand}{\enspace \text{and} \enspace}
\newcommand{\ton}{\enspace \text{on} \enspace}
\newcommand{\tfor}{\enspace \text{for} \enspace}
\newcommand{\tin}{\enspace \text{in} \enspace}
\newcommand{\tas}{\enspace \text{as} \enspace}
\newcommand{\twith}{\enspace \text{with} \enspace}

\newcommand{\rd}{\mathrm{d}}

\renewcommand{\rm}{\mathrm{m}}
\newcommand{\rr}{\mathrm{r}}
\newcommand{\rv}{\mathrm{v}}
\newcommand{\rvs}{\mathrm{vs}}
\newcommand{\rc}{\mathrm{c}}
\newcommand{\rev}{\mathrm{ev}}
\newcommand{\rcd}{\mathrm{cd}}
\newcommand{\rac}{\mathrm{ac}}
\newcommand{\rcr}{\mathrm{cr}}
\newcommand{\rth}{\mathrm{th}}
\newcommand{\rcp}{\mathrm{cp}}
\newcommand{\rref}{\mathrm{ref}}
\newcommand{\rpd}{\mathrm{pd}}
\newcommand{\rpv}{\mathrm{pv}}
\newcommand{\rcn}{\mathrm{cn}}

\newcommand{\trho}{\Tilde{\rho}}
\newcommand{\tu}{\Tilde{u}}

\newcommand{\tz}{\Tilde{z}}
\newcommand{\tT}{\Tilde{T}}
\newcommand{\tq}{\Tilde{q}}
\newcommand{\tp}{\Tilde{p}}
\newcommand{\tV}{\Tilde{V}}
\newcommand{\tQ}{\Tilde{Q}}

\begin{document}

\title[Moisture-compressible NSE]{Moisture dynamics with phase changes coupled to heat-conducting, compressible fluids}

\author{Felix Brandt}
\address{Department of Mathematics, University of California, Berkeley, Berkeley, CA 94720, USA.}
\email{fbrandt@berkeley.edu}

\author{Matthias Hieber}
\address{Technische Universit\"{a}t Darmstadt,
Schlo\ss{}gartenstra{\ss}e 7, 64289 Darmstadt, Germany.}
\email{hieber@mathematik.tu-darmstadt.de}

\author{Lin Ma}
\address{School of Mathematical Sciences, Capital Normal University, Beijing
100048, P. R. China}
\email{malin.cnu@foxmail.com}

\author{Tarek Z\"{o}chling}
\address{Technische Universit\"{a}t Darmstadt,
Schlo\ss{}gartenstra{\ss}e 7, 64289 Darmstadt, Germany.}
\email{zoechling@mathematik.tu-darmstadt.de}
\subjclass[2020]{35K59, 35Q35, 35Q86, 76N10}
\keywords{Strong well-posedness; Lagrangian coordinates; optimal $\rL^p$-$\rL^q$ estimates; moist atmospheric flows; micro-physics moisture model; phase changes; compressible non-isothermal flows.}

\begin{abstract}
It is shown that a model coupling the heat-conducting compressible Navier-Stokes equations to a micro-physics model of moisture in air is locally strongly well-posed for large data in suitable function spaces and strongly well-posed on $[0,\tau]$ for every $\tau > 0$ for small initial data.
This seems to be the first result on $[0,\tau]$ for arbitrary $\tau > 0$ for a model coupling moisture dynamics to heat-conducting, compressible Navier-Stokes equations.
A key feature of the micro-physics model is that it also includes phase changes of water in moist air.
These phase changes are associated with large amounts of latent heat and thus result in a strong coupling to the thermodynamic equation.
The well-posedness results are obtained by means of a Lagrangian approach, which allows to treat the hyperbolicity in the continuity equation.
More precisely, optimal $\rL^p$-$\rL^q$ estimates are shown for the linearized system, leading to the local well-posedness result by a fixed point argument and suitable nonlinear estimates.
For the well-posedness result on $[0,\tau]$ for arbitrary $\tau > 0$, a refined analysis of the linearized problem close to equilibria is carried out, and the roughness of the source term, induced by the phase changes, requires to establish delicate a priori bounds.
\end{abstract}

\maketitle

\section{Introduction}\label{sec:intro}

Modeling and understanding moisture dynamics in the atmosphere is important in fluid dynamics and atmospheric sciences.
Due to the complex nature of the moisture processes in the atmosphere, which in turn results from the coupling of the thermodynamics of water and dry air, this is a challenging task, and there are various approaches by different communities.
For background on moisture models, we refer for instance to the survey article of Stevens\cite{Ste:05}, the article by Pauluis, Czaja and Korty\cite{PCK:08} as well as the monograph by Khouider\cite{Kho:19}, and the references therein.
There is also vast literature on the mathematical analysis of moisture models.
For the case of only one moisture quantity, we refer to the works by Bousquet, Coti-Zelati, Fr\'emond, Huang, Kukavica, Temam, Tribbia and Ziane\cite{BCZT:14,CZFTT:13,CZHKTZ:15,CZT:12}. 
In \cite{LT:16}, Li and Titi established the global well-posedness and a relaxation limit of a variant of a model of large scale dynamics of precipitation fronts in the tropical atmosphere first introduced by Frierson, Majda and Pauluis\cite{FMP:04}.
A similar problem was also studied by Majda and Souganidis\cite{MS:10} in the weak setting.
Zhang, Smith and Stechmann\cite{ZSS:21} numerically investigated the effects of clouds and phase changes on fast-wave averaging.
Very recently, Remond-Tiedrez, Smith and Stechmann\cite{RTSS:24a} showed the well-posedness and regularity at the cloud edge for a nonlinear elliptic PDE arising in the context of the precipitating quasi-geostrophic model.
A moist Boussinesq system was then investigated by the same authors in \cite{RTSS:24b}.

For the underlying moisture model in this paper, the starting point is that water in the atmosphere is present in the form of vapor water, cloud water and rain water.
Importantly, by phase changes, these various forms can be transformed into one another.
More precisely, rain water evaporates to vapor water, condensation leads to the change of vapor water into cloud water, while accumulation and collection of cloud water by rain result in the formation of rain water.
The effects of these phase changes are twofold:
First, they lead to a strong coupling of the moisture quantities to the thermodynamic equation, also owing to the presence of different heat capacities for dry air, water vapor and liquid water.
Second, these phase changes cause significant mathematical challenges due to switch terms in the source terms, which thus lack Fr\'echet differentiability.
We can overcome this mathematical difficulty by our approach based on a combination of classical energy estimates with more modern tools as optimal data results for compressible fluids.
This allows us to regard the switch terms, which are non-differentiable, as forcing terms.

From a modeling point of view, based on the works of Kessler\cite{Kes:69} and Grabowski and Smolarkiewicz\cite{GS:96}, Klein and Majda\cite{KM:06} developed a moisture model for warm clouds, corresponding to a bulk micro-physics closure.
The model from \cite{KM:06} was later on refined by Hittmeir and Klein\cite{HK:18}, who considered a more sophisticated thermodynamic setting, giving rise to a significantly more strongly coupled problem.

The aim of this paper is to rigorously study a moist atmosphere dynamics model that is coupled to the heat-conducting compressible Navier-Stokes equations whose analysis was pioneered by Matsumura and Nishida\cite{MN:80,MN:83}.
We refer here also to the works of Lions\cite{Lio:98} and Feireisl\cite{Fei:04} in the weak setting.
For the existence of global weak solutions for viscous compressible and heat-conducting Navier-Stokes equations, we refer to the article by Bresch and Desjardins\cite{BD:07}.
A multi-scale analysis of compressible viscous and rotating fluids has been carried out by Feireisl, Gallagher, Gerard-Varet and Novotn\'y\cite{FGGVN:12}.
Danchin\cite{Dan:14} used a Lagrangian approach to the compressible Navier-Stokes equations.
For further background on Fourier analysis methods, including applications to the compressible Navier-Stokes equations, we also refer to the monograph of Bahouri, Chemin and Danchin\cite{BCD:11} and the references therein.

Let us emphasize that we do not consider a Boussinesq approximation, which would be simpler to analyze, but we take into account the full non-isothermal compressible Navier-Stokes equations, comprising a continuity equation for the density of dry air, a momentum balance as well as a diffusion equation for the temperature.

Next, we briefly describe the resulting model coupling the moisture model from \cite{HK:18} to the non-isothermal compressible Navier-Stokes equations.
We consider the coupled model of moisture dynamics and compressible Navier-Stokes equations on a domain of the shape $\Omega= G \times (0,1)$, with~$G = (0,1)^2$.
Moreover, we denote by $\tau \in (0,\infty]$ a positive time.
Then $\rho_\rd$, $\rho_\rv$, $\rho_\rc$ and $\rho_\rr$ represent the densities of dry air, vapor water, cloud water and rain water, respectively.
With these quantities, the mixing ratios for vapor water, cloud water and rain water are defined by $q_j \coloneqq \nicefrac{\rho_j}{\rho_\rd}$, where $j \in \{\rv,\rc,\rr\}$.
Another important variable is the air velocity of the atmosphere $u = (v,w)$, where $v$ is the horizontal velocity, while $w$ represents the vertical velocity.
In this context, $p$, $\theta$ and $T$ denote the pressure, the potential temperature and the thermodynamic temperature, respectively.
Besides, $V_\rr$ is the sedimentation velocity, which is assumed to be a smooth function, while $Q_\rm$, $Q_\rth$ and $Q_\rcp$ depend on the moisture quantities.
Furthermore, $Q_1$, $Q_2$, $\mu$, $\lambda$, $g$, $\kappa$, $c_1$, $R_\rd$ and $R_\rv$ are suitable constants.
Let us also briefly elaborate on the source terms.
By $S_\rev$, we denote the rate of evaporation of rain water, $S_\rcd$ represents the condensation of vapor water to cloud water and the inverse process, $S_\rac$ accounts for the auto-conversion of cloud water into rain water by accumulation of microscopic droplets, and $S_\rcr$ is the collection of cloud water by falling rain.
Note that these terms will be made precise in \autoref{sec:coupled model}.
For the convective derivative $\rD_t = \partial_t + u \cdot \nabla$, the model coupling the moisture dynamics and non-isothermal compressible Navier-Stokes equations then reads as 
\begin{equation}\label{eq:coupled moisture compr NSE comp}
    \left\{
    \begin{aligned}
        \rD_t\rho_{\rd}+\rho_{\rd} \mdiv u
        &=0,\\
        \rho_\rd Q_\rm\rD_t u-\mu\Delta u-(\mu+\lambda)\nabla\mdiv u+\nabla p
        &=\rho_{\rd}q_{\rr}V_{\rr}\partial_{z}u-\rho_{\rd}Q_{\rm}g \mre_{3},\\
        Q_{\rth} \rD_t T-\kappa \Delta T
        &=c_1 q_{\rr}V_{\rr}\partial_{z}T + Q_{\rcp} T \mdiv u- (Q_1 T + Q_2)(S_\rev - S_\rcd),\\
        \rD_t q_{\rv}-\Delta q_{\rv} 
        &= S_\rev - S_\rcd,\\
		\rD_t q_{\rc}-\Delta q_{\rc}
        &= S_\rcd - (S_\rac + S_\rcr),\\
		\rD_t q_{\rr}-\Delta q_{\rr}
        &= \del_z(q_\rr V_\rr) + q_\rr V_\rr \del_z \log \rho_\rd + (S_\rac + S_\rcr) - S_\rev,\\
		p
        &=\rho_{\rd}(R_\rd+R_\rv q_{\rv}) T.
    \end{aligned}
    \right.
\end{equation}
The model is complemented by suitable boundary and initial conditions, see Eqs.~\eqref{eq:bc fluid}--\eqref{eq:init conds}.
In this article, we establish the local strong well-posedness as well as the global strong well-posedness for small data of the model presented in \eqref{eq:coupled moisture compr NSE comp}.
In order to handle the hyperbolic effects from the continuity equation in \eqref{eq:coupled moisture compr NSE comp}$_1$, we first transform the coupled system to Lagrangian coordinates.
We then develop an $\rL^p$-$\rL^q$-theory to the linearized problem, showing the existence of a unique solution to the linearized problem, and providing a priori estimates in terms of the data.
Thanks to a careful analysis of the transformed terms, this leads to the local strong well-posedness by means of a fixed point procedure.
Concerning the global well-posedness for small data, we slightly adjust the transformation to Lagrangian coordinates in order to investigate the system close to equilibrium solutions.
A refined analysis of the linearized operator associated with the compressible Navier-Stokes moisture system, together with suitable nonlinear estimates to handle the rough source terms, then paves the way for the second main result of this paper.
Major difficulties here are the lack of Fr\'echet differentiability of the nonlinear terms, which is due to the switch terms that account for the phase changes, and the non-locality in time of the transformed terms.
The special shape of the source terms requires us to use an approach combining maximal $\rL^p$-regularity with energy estimates.
As far as we can see, only using either maximal $\rL^p$-regularity or energy estimates is insufficient to handle the switch terms for the global strong well-posedness for small data.
For more background on maximal $\rL^p$-regularity and optimal $\rL^p$-$\rL^q$ estimates, we refer e.g., to the monograph\cite{DHP:03} and the article\cite{DHP:07}, while details on the underlying function spaces can for instance be found in the monographs\cite{Ama:19,Tri:78}.

The rigorous analytical investigation of models as introduced in \cite{HK:18,KM:06} started only recently.
In fact, in \cite{HKLT:17}, Hittmeir, Klein, Li and Titi showed that for a prescribed velocity field, the model from~\cite{KM:06} is globally strongly well-posed.
The same authors extended the analysis to the situation of the moisture model coupled to the viscous, incompressible primitive equations in \cite{HKLT:20}, where, building on the groundbreaking global strong well-posedness result for the primitive equations due to Cao and Titi\cite{CT:07}, the global solvability of the resulting coupled model was obtained.
The refined model from \cite{HK:18} was then analyzed in \cite{HKLT:23}, and the authors established the global strong well-posedness in the situation of a prescribed fluid velocity.
In a very recent paper, Doppler, Klein, Liu and Titi\cite{DKLT:24} showed the local strong well-posedness of the model coupling the moisture model from \cite{HK:18} with the compressible Navier-Stokes equations as presented in short form in \eqref{eq:coupled moisture compr NSE comp} for initial data in $\rH^2$.
In this paper, we extend the result from \cite{DKLT:24} in two directions, namely we establish its local strong well-posedness for a larger class of initial data, and we prove the (almost) global strong well-posedness for small data, see also \autoref{sec:main results} for the precise results.

This article is organized as follows.
In \autoref{sec:coupled model}, we expand on the coupled model comprising the moisture dynamics as well as the compressible Navier-Stokes equations.
\autoref{sec:main results} is dedicated to the presentation of the main results in the $\rL^p$-$\rL^q$-setting.
We then transform the coupled model to Lagrangian coordinates in \autoref{sec:Lagrangian coords}.
In \autoref{sec:lin theory}, we discuss the solvability of the linearized problem and prove optimal $\rL^p$-$\rL^q$ estimates.
Together with nonlinear estimates, this leads to the proof of the local strong well-posedness result in \autoref{sec: local}.
\autoref{sec: global} is devoted to the proof of the global well-posedness for small data.

\section{Description of the Coupled Moisture-Fluid Model}\label{sec:coupled model}

The purpose of this section is to introduce the coupled model of the moisture dynamics and non-isothermal compressible Navier-Stokes equations in detail.
To this end, let us briefly recall that we consider the problem on a unit box with horizontal periodicity $\Omega = G \times (0,1)$, where $G = (0,1)^2$, and $\tau \in (0,\infty]$ represents a positive time.
Additionally, we recall the densities of dry air $\rho_\rd$, vapor water $\rho_\rv$, cloud water~$\rho_\rc$ and rain water $\rho_\rr$ as well as the resulting mixing ratios for vapor water $q_\rv = \nicefrac{\rho_\rv}{\rho_\rd}$, cloud water $q_\rc = \nicefrac{\rho_\rc}{\rho_\rd}$ and rain water $q_\rr = \nicefrac{\rho_\rr}{\rho_\rd}$.
The moist air density~$\rho$ then takes the shape
\begin{equation*}
    \rho \coloneqq \rho_\rd(1 + q_\rv + q_\rc + q_\rr).
\end{equation*}
For the potential temperature $\theta$ as well as positive constants $p_\rref$, $c_\rpd$ and $R_\rd$, with $c_\rpd > R_\rd$, the temperature $T$ is given by
\begin{equation*}
    T \coloneqq \theta \Bigl(\frac{p}{p_\rref}\Bigr)^{\frac{\gamma - 1}{\gamma}}, \enspace \text{where} \enspace \gamma = \frac{c_\rpd}{c_\rpd - R_\rd} > 1.
\end{equation*}
Moreover, for further constants $c_\rpv$, $c_1 > 0$, the mixed specific heat capacity $c_\nu$ is of the form
\begin{equation*}
    c_\nu \coloneqq c_\rpd + c_\rpv q_\rv + c_1(q_\rc + q_\rr),
\end{equation*}
while for another positive constant $R_\rv$, the mixed gas constant $\sigma$ is defined by
\begin{equation*}
    \sigma \coloneqq \Bigl(\frac{c_\rpv}{c_\rpd} R_\rd - R_\rv\Bigr) q_\rv + \frac{c_1}{c_\rpd} R_\rd(q_\rc + q_\rr),
\end{equation*} 
and for positive reference values $L_\rref$ and $T_\rref$, the latent heat of condensation is
\begin{equation*}
    L(T) \coloneqq L_\rref + (c_\rpv - c_1)(T - T_\rref).
\end{equation*}
Another important quantity is the saturation mixing ratio $q_\rvs \coloneqq q_\rvs(p,T)$, which is assumed to be a non-negative, uniformly bounded and Lipschitz continuous function of the pressure $p$ and the temperature~$T$, and we suppose that $0 \le q_\rvs \le q_\rvs^* < \infty$.
Here $f^+$ and $f^-$ represent the positive and the negative part of a function $f$, i.e., $f^+ \coloneqq \max(0,f)$ and $f^- \coloneqq \min(0,f)$.
With the above quantities, we have
\begin{equation*}
    \begin{aligned}
        Q_\rm
        &\coloneqq 1 + q_\rv + q_\rc + q_\rr,\\
        Q_\rth
        &\coloneqq \frac{c_\nu}{\gamma} + \sigma = \frac{c_\rpd}{\gamma} + \Bigl(\frac{c_\rpv}{\gamma} + \frac{c_\rpv}{c_\rpd} R_\rd - R_\rv\Bigr) q_\rv + \Bigl(\frac{c_1}{\gamma} + \frac{c_1}{c_\rpd}R_\rd\Bigr)(q_\rc + q_\rr),\\
        Q_\rcp
        &\coloneqq \sigma - \frac{R_\rd}{c_\rpd} c_\nu = - R_\rd - R_\rv q_\rv,\\
        Q_1
        &\coloneqq \frac{R_\rv}{R_\rd + R_\rv q_\rv}\Bigl(\sigma - \frac{R_\rd}{c_\rpd} c_\nu\Bigr) + c_\rpv - c_1 = c_\rpv - c_1 - R_\rv, \tand\\
        Q_2 
        &\coloneqq L_\rref - (c_\rpv - c_1) T_\rref.
    \end{aligned}
\end{equation*}

Next, we make precise the source terms in \eqref{eq:coupled moisture compr NSE comp}.
In fact, these are the rates of evaporation of rain water~$S_\rev$, the condensation of vapor water to cloud water and the inverse process $S_\rcd$, the auto-conversion of cloud water into rain water by accumulation of microscopic droplets $S_\rac$ and the collection of cloud water by falling rain $S_\rcr$.
They are given by
\begin{equation*}
    \begin{aligned}
        S_\rev
        &\coloneqq c_\rev \frac{p}{\rho}(q_\rvs - q_\rv)^+ q_\rr = c_\rev T \frac{R_\rd + R_\rv q_\rv}{1 + q_\rv + q_\rc + q_\rr}(q_\rvs - q_\rv)^+ q_\rr,\\
        S_\rcd
        &\coloneqq c_\rcd(q_\rv - q_\rvs) q_\rc + c_\rcn(q_\rv - q_\rvs)^+ q_\rcn,\\
        S_\rac
        &\coloneqq c_\rac(q_\rc - q_\rac)^+, \tand S_\rcr \coloneqq c_\rcr q_\rc q_\rr.
    \end{aligned}
\end{equation*}
Here $c_\rev$, $c_\rcd$, $c_\rcn$, $c_\rac$, $q_\rcn$ and $q_\rac$ denote further positive constants.
Let us remark that for simplicity of the presentation, and as it does not affect the analysis, we assume that the constants are such that
\begin{equation*}
    \begin{aligned}
        T
        &= \theta p^{\frac{\gamma-1}{\gamma}}, \enspace c_\nu = 1 + q_\rv + q_\rc + q_\rr, \enspace \sigma = q_\rv + q_\rc + q_\rr, \enspace L(T) = L_\rref + T - T_\rref,\\
        Q_\rm 
        &= Q_\rth = 1 + q_\rv + q_\rc + q_\rr, \enspace Q_\rcp = -(1 + q_\rv), \enspace c_1 = q_\rcn = q_\rac = Q_1 = Q_2 = 1,\\
        S_\rev
        &= T \cdot \frac{1 + q_\rv}{Q_\rm}(q_\rvs - q_\rv)^+ q_\rr, \enspace S_\rcd = (q_\rv - q_\rvs) q_\rc + (q_\rv - q_\rvs)^+,\\
        S_\rac 
        &= (q_\rc - 1)^+ \tand S_\rcr = q_\rc q_\rr,
    \end{aligned}
\end{equation*}
and we suppose that $\kappa = 1$.
For further details and an explanation of the constants, we also refer to~\cite[Sec.~1]{DKLT:24}.

Plugging in the terms resulting from the aforementioned simplifications, and defining the Lam\'e operator~$\rL$ by~$\rL u \coloneqq \mu\Delta u +(\mu+\lambda)\nabla\mdiv u$, where we assume~$\mu > 0$ and $2\mu + \lambda > 0$, we arrive at the following system of equations on~$(0,\tau) \times \Omega$:
\begin{equation}\label{eq:coupled moisture compr NSE detailed}
    \left\{
    \begin{aligned}
        \del_t \rho_{\rd} + \mdiv(\rho_\rd u)
        &=0,\\
        \rho_{\rd}Q_{\rm}[\del_t u + (u \cdot \nabla)u]-\rL u+\nabla p
        &=\rho_{\rd}q_{\rr}V_{\rr}\partial_{z}u -\rho_{\rd}Q_{\rm}g \mre_{3},\\
        Q_{\rm}[\del_t T + (u \cdot \nabla)T]-\Delta T
        &=q_{\rr}V_{\rr}\partial_{z}T-(1+q_{\rv})T \mdiv u-(T+1)\\
        &\quad \cdot \Bigl[T \cdot \frac{1 + q_\rv}{Q_\rm}(q_\rvs - q_\rv)^+ q_\rr- (q_\rv - q_\rvs) q_\rc - (q_\rv - q_\rvs)^+\Bigr],\\
        \del_t q_{\rv} + (u \cdot \nabla) q_{\rv}-\Delta q_{\rv} 
        &=T \cdot \frac{1 + q_\rv}{Q_\rm}(q_\rvs - q_\rv)^+ q_\rr- (q_\rv - q_\rvs) q_\rc - (q_\rv - q_\rvs)^+,\\
        \del_t q_{\rc} + (u \cdot \nabla) q_{\rc}-\Delta q_{\rc}
        &= (q_\rv - q_\rvs) q_\rc + (q_\rv - q_\rvs)^+ - (q_\rc - 1)^+ - q_\rc q_\rr ,\\
        \del_t q_{\rr} + (u \cdot \nabla) q_{\rr}-\Delta q_{\rr}
        &= \frac{1}{\rho_{\rd}}\partial_{z}(\rho_{\rd} q_{\rr}V_{\rr}) + (q_{\rc}-1)^{+}+q_{\rc}q_{\rr} - T \cdot \frac{1 + q_\rv}{Q_\rm}(q_\rvs - q_\rv)^+ q_\rr,\\
		p
        &=\rho_{\rd}(1+q_{\rv}) T.
    \end{aligned}
    \right.
\end{equation}
The model is completed by suitable boundary conditions.
More precisely, we assume that all quantities satisfy periodic boundary conditions on the lateral boundary.
Let us denote the upper and lower boundary by $\Gau \coloneqq G \times \{1\}$ and $\Gal \coloneqq G \times \{0\}$. For the fluid velocity, we assume no-slip boundary conditions, i.e.,
\begin{equation}\label{eq:bc fluid}
    u = 0, \ton \Gau \cup \Gal,
\end{equation}
and we assume homogeneous Neumann boundary conditions for the other quantities, so
\begin{equation}\label{eq:bc heat & moisture}
    \del_z T = 0 \tand \del_z q_j = 0, \tfor j \in \{\rv,\rc,\rr\}, \ton \Gau \cup \Gal.
\end{equation}
Finally, we also take into account initial conditions
\begin{equation}\label{eq:init conds}
    \begin{aligned}
        \rho_\rd(0) 
        &= \rho_{\rd,0}, \enspace u(0) = u_0, \enspace T(0) = T_0,\\
        q_\rv(0) 
        &= q_{\rv,0}, \enspace q_\rc(0) = q_{\rc,0} \tand q_\rr(0) = q_{\rr,0},
    \end{aligned}
\end{equation}
which are required to satisfy the boundary conditions \eqref{eq:bc fluid} and \eqref{eq:bc heat & moisture} provided they are assumed to be sufficiently regular as in \autoref{thm:localWP} or \autoref{thm: global WP}.
In this context, let us remark that we use $Q_{\rm,0}$ in order to denote $Q_\rm(0) = 1 + q_{\rv,0} + q_{\rc,0} + q_{\rr,0}$.

\section{Main Results}\label{sec:main results}

Having introduced the complete model in detail, we state the main results on the local strong well-posedness as well as the global strong well-posedness for small data in the $\rL^p$-$\rL^q$-setting in this section.
With regard to strong solutions for the compressible Navier-Stokes equations, it seems that there are basically two approaches:
One is based on a Hilbert space setting and relies on higher Sobolev regularity for the initial data.
In addition to more restrictive regularity requirements for the initial data, this also leads to more complicated compatibility conditions that need to be satisfied by the initial data.
The second approach, which we mainly pursue in this paper, does not require to consider higher regularity.
Instead, for sufficiently large~$q$, the ground space for the fluid velocity, temperature equation and mixing ratios is $\rL^q$, while we take into account $\rH^{1,q}$ for the continuity equation.
Instead, for sufficiently large $q$, we choose $\rL^q$ as the base space for the fluid velocity, temperature equation, and mixing ratios, while the continuity equation is treated in the space~$\rH^{1,q}$.

Before stating the main results, we need to clarify some of the function spaces used throughout this paper.
Let $s \in \R$ and $p,q \in (1,\infty)$. In the sequel, we
denote by $\rW^{s,q}(\Omega)$ the fractional Sobolev spaces, by~$\rH^{s,q}(\Omega)$ the Bessel potential spaces and  by $\rB^{s}_{q,p}(\Omega)$ the Besov spaces, where $\Omega \subset \R^3$ is a bounded domain.
We set $\rH^{0,q}(\Omega) \coloneqq \rL^q(\Omega)$ and note that $\rH^{s,q}(\Omega)$ coincides with the 
classical Sobolev space~$\rW^{m,q}(\Omega)$ provided  $s=m \in \N$. 
In the following, we need some terminology to describe periodic boundary conditions on $\Gamma_l = \partial G \times [0,1]$ as well as on $\partial G$. 
Given $s \in [0,\infty)$ and~$p,q \in (1,\infty)$, we define the spaces $\rH^{s,q}_\per (\Omega) \coloneqq \overline{\rC^\infty_\per (\overline{\Omega})}^{\| \cdot \|_{\rH^{s,q}(\Omega)}}$ and $\rH^{s,q}_\per (G) \coloneqq \overline{\rC^\infty_\per (\overline{G})}^{\| \cdot \|_{\rH^{s,q}(G)}}$, where horizontal periodicity is modeled by the function spaces $\rC^\infty_\per(\overline{\Omega})$ and $\rC^\infty_\per(\overline{G})$ defined e.g., in \cite{HK:16}. 
The Besov spaces $\rB^s_{q,p,\per}(\Omega)$ and~$\rB^s_{q,p,\per}(G)$ are defined likewise. 
In order to shorten notation, we abbreviate the norms by a subscript~$_\tau$ for the time component and a subscript $_x$ for the spatial component, i.e., instead of~$\| \cdot \|_{\rL^p(0,\tau;\rL^q(\Omega))}$, we write $\| \cdot \|_{\rL_{\tau}^p(\rL_x^q)}$, and likewise in the situation of Sobolev or Besov spaces.

We are now in the position to formulate the first main result on the local, strong well-posedness of \eqref{eq:coupled moisture compr NSE detailed} for large data.

\begin{thm}[Local strong well-posedness]  \mbox{} \\
    \label{thm:localWP}
    Let $\frac{2}{p}+\frac{3}{q}<1$.
    Assume that 
    \begin{equation*}
        \begin{aligned}
            (\rho_{\rd,0},u_0, T_0, q_{j,0}) \in \rH^{1,q}_\per(\Omega) \times  \rB_{q,p,\per}^{2(1-\nicefrac{1}{p})}(\Omega)^3 \times \rB_{q,p,\per}^{2(1-\nicefrac{1}{p})}(\Omega) \times \rB_{q,p,\per}^{2(1-\nicefrac{1}{p})}(\Omega) \eqqcolon \rY_\gamma, 
        \end{aligned}
    \end{equation*}
    for $j\in \{ \rv,\rc,\rr\}$, such that $u_0 = 0$ and $\dz T_0 = \dz q_j = 0$ on $\Gau \cup \Gal$, for $j \in \{\rv,\rc,\rr\}$, and suppose that $(\rho_{\rd,0},q_{j,0})$ fulfill
    \begin{equation*}
        M_1\leq \rho_{\rd,0},q_{j,0}\leq M_2 \text{ in }\Omega, \ \text{ for some constants }M_1,M_2>0.
    \end{equation*}
Then there exists $a = a(\rho_{\rd,0},u_0, T_0, q_{j,0}) > 0$, such that the system \eqref{eq:coupled moisture compr NSE detailed} admits a unique, strong solution $(\rho_{\rd},u,T,q_\rv,q_\rc,q_\rr)$ satisfying
  \begin{equation*}
        \begin{aligned}
            \rho_{\rd} \in \rH^{1,p}(0,a;\rH^{1,q}_\per(\Omega)), \ u 
            &\in \rH^{1,p}(0,a;\rLq(\Omega)^3) \cap \rLp(0,a;\rH^{2,q}_\per(\Omega)^3), \tand \\ T,q_j 
            &\in \rH^{1,p}(0,a;\rLq(\Omega)) \cap \rLp(0,a;\rH^{2,q}_\per(\Omega)), \tfor j \in \{ \rv,\rc,\rr\}.
        \end{aligned}
    \end{equation*}
    Moreover, there exist constants $M_1^*,M_2^*>0$ such that $M_1^*\leq \rho_{\rd}\leq M_2^*$ for all $t \in (0,a)$ and $j \in \{ \rv,\rc,\rr\}$.
\end{thm}

In order to state the well-posedness result on $[0,\tau]$ for arbitrary $\tau > 0$, we assume that $q_\rvs \equiv 0$. 
It is straightforward to check that $(\overline{\rho}_\rd, 0, \overline{T}, 0, 1, 0 )$ is an equilibrium solution to the coupled problem.

\begin{thm}[\text{Well-posedness on $[0,\tau]$ for large $\tau > 0$ for small data}] \mbox{} \\
    \label{thm: global WP}
     Let $\tau>0$ be an arbitrary given time, and let $\frac{2}{p}+\frac{3}{q}<1$.
     Moreover, suppose that $(\rho_{\rd,0},u_0, T_0, q_{j,0}) \in \rY_\gamma$, for $j\in \{ \rv,\rc,\rr\}$, so that $u_0 = 0$ and $\dz T_0 = \dz q_j = 0$ on~$\Gau \cup \Gal$, for $j \in \{\rv,\rc,\rr\}$, and in addition that
     \begin{equation*}
         \| ( \rho_{\rd,0} - \overline{\rho}_\rd, u_0, T_0 - \overline{T}, q_{\rv,0}, q_{\rc,0}-1, q_{\rr,0} ) \|_{\rH^{1,q}(\Omega) \times  \rB_{q,p}^{2(1-\nicefrac{1}{p})}(\Omega)^5} \leq \delta,
     \end{equation*}
     for some $\delta>0$ and constants $\overline{\rho}_\rd>0$, $\overline{T} \geq0$. 
     Then the system \eqref{eq:coupled moisture compr NSE detailed} admits a unique, strong solution $(\rho_{\rd},u,T,q_\rv,q_\rc,q_\rr)$ on $[0,\tau]$ satisfying
    \begin{equation*}
        \begin{aligned}
            &\| \rho_\rd - \overline{\rho}_\rd \|_{\rL^\infty(0,\tau;\rH^{1,q}(\Omega))} + \| \nabla \rho_\rd \|_{\rH^{1,p}(0,\tau;\rLq(\Omega))} + \| \dt \rho_\rd \|_{\rLp(0,\tau;\rH^{1,q}(\Omega))} \\
            &+ \| u \|_{\rLp(0,\tau;\rH^{2,q}(\Omega))}+ \| \dt u \|_{\rLp(0,\tau;\rLq(\Omega))}+ \| u \|_{\rL^\infty(0,\tau;\rB^{2(1-\nicefrac{1}{p})}_{q,p}(\Omega))} \\ 
            &+\| T-\overline{T} \|_{\rLp(0,\tau;\rH^{2,q}(\Omega))}+ \| \dt T \|_{\rLp(0,\tau;\rLq(\Omega))}+ \| T-\overline{T} \|_{\rL^\infty(0,\tau;\rB^{2(1-\nicefrac{1}{p})}_{q,p}(\Omega))} \\ 
            & + \sum_{j \in \{ \rv,\rr\}} \| q_j \|_{\rLp(0,\tau;\rH^{2,q}(\Omega))}+ \| \dt q_j \|_{\rLp(0,\tau;\rLq(\Omega))}+ \| q_j \|_{\rL^\infty(0,\tau;\rB^{2(1-\nicefrac{1}{p})}_{q,p}(\Omega))} \\
            & +\| q_\rc-1 \|_{\rLp(0,\tau;\rH^{2,q}(\Omega))}+ \| \dt q_\rc \|_{\rLp(0,\tau;\rLq(\Omega))}+ \| q_\rc-1 \|_{\rL^\infty(0,\tau;\rB^{2(1-\nicefrac{1}{p})}_{q,p}(\Omega))}  \leq C\delta,
        \end{aligned}
    \end{equation*}
    for a constant $C>0$. Moreover, $\rho_\rd(t,x) \geq \nicefrac{\overline{\rho}_\rd}{2}$ for all $t \in [0,\tau]$ and $x \in \Omega$.
\end{thm}

\begin{rem}\label{rem_ assu}
    { Concerning the assumptions on $q_\rvs$ and the sedimentation velocity~$V_\rr$, we note that any constant $q_\rvs^*\geq0$ is permissible, and we can consider any given function $V_\rr$ which is sufficiently smooth.
    However, for convenience of notation, we focus on the case $q_\rvs \equiv 0$ and $V_\rr \equiv 1$.} 
\end{rem}

\section{The Model in Lagrangian Coordinates}\label{sec:Lagrangian coords}

In order to circumvent the hyperbolic effects that arise in the continuity equation, we introduce Lagrangian coordinates.
The present section is dedicated to the calculation of the associated change of variables.

For this purpose, we introduce the characteristics $X$ that corresponds to the fluid velocity $u$, i.e., $X$ solves the Cauchy problem
\begin{equation}\label{eq:CP Lagrange}
    \left\{
    \begin{aligned}
        \del_t X(t,y)
        &= u(t,X(t,y)), &&\tfor t \in (0,\tau),\\
        X(0,y)
        &= y, &&\tfor y \in \Omega.
    \end{aligned}
    \right.
\end{equation}
Assume that $X(t,\cdot)$ is a $\rC^1$-diffeomorphism from $\Omega$ onto $\Omega$ for all $t \in (0,\tau)$, and for all $t \in (0,\tau)$, denote by $Y(t,\cdot) = [X(t,\cdot)]^{-1}$ the inverse of $X(t,\cdot)$.
The (local-in-time) invertibility of $X$ will be discussed in \autoref{lem:ests of trafo}(i) below.
For~$(t,y) \in (0,\tau) \times \Omega$, we consider the change of variables given by
\begin{equation*}
    \begin{aligned}
        \rho_\rd(t,X(t,y))
        &= \trho_\rd(t,y), \enspace u(t,X(t,y)) = \tu(t,y), \enspace T(t,X(t,y)) = \tT(t,y),\\
        q_j(t,X(t,y))
        &= \tq_j(t,y), \tfor j \in \{\rv,\rc,\rr\},\\
        p(t,X(t,y))
        &= \rho_\rd(t,X(t,y))(1 + q_\rv(t,X(t,y)))T(t,X(t,y))\\
        &= \trho_\rd(t,y)(1 + \tq_\rv(t,y))\tT(t,y) = \tp(t,y),\\
        V_\rr(t,X(t,y))
        &= \tV_\rr(t,y), \enspace q_\rvs(p(t,X(t,y)),T(t,X(t,y))) = q_\rvs(\tp(t,y),\tT(t,y)) \eqqcolon \tq_\rvs.
    \end{aligned}
\end{equation*}
In particular, for $(t,x) \in (0,\tau) \times \Omega$, we have
\begin{equation*}
    \begin{aligned}
        \rho_\rd(t,x)
        &= \trho_\rd(t,Y(t,x)), \enspace u(t,x) = \tu(t,Y(t,x)), \enspace T(t,x) = \tT(t,Y(t,x)),\\
        q_j(t,x)
        &= \tq_j(t,Y(t,x)), \tfor j \in \{\rv,\rc,\rr\},\\
        p(t,x)
        &= \rho_\rd(t,x)(1 + q_\rv(t,x))T(t,x)\\ 
        &= \trho_\rd(t,Y(t,x))(1 + \tq_\rv(t,Y(t,x)))\tT(t,Y(t,x))= \tp(t,Y(t,x)),\\
        V_\rr(t,x)
        &= \tV_\rr(t,Y(t,x)), \enspace q_\rvs(p(t,x),T(t,x)) = q_\rvs(\tp(t,Y(t,x)),\tT(t,Y(t,x))) \eqqcolon \tq_\rvs.
    \end{aligned}
\end{equation*}
We denote by $Z$ the inverse of the gradient of $X$, i.e., $Z(t,y) = Z_{i,j} = [\nabla X]^{-1}(t,y)$ for $t \in (0,\tau)$ and $y \in \Omega$.
Moreover, we will also write $\tQ_\rm = 1 + \tq_\rv + \tq_\rc + \tq_\rr$.

With this change of coordinates, lengthy but straightforward calculations lead to the following system of equations, where we write the original linear terms on the left-hand side in order to prepare for the nonlinear estimates.
In fact, the transformed system of equations associated to \eqref{eq:coupled moisture compr NSE detailed} and in Lagrangian coordinates reads
\begin{equation}\label{eq:Lagrangian coords}
	\left\{	
    \begin{aligned}
        \del_t\trho_\rd+\rho_{\rd,0}\mdiv\tu
        &=G_\rd, &&\tin (0,\tau) \times \Omega,\\
		\del_t\tu-\frac{1}{\rho_{\rd,0}Q_{\rm,0}}\rL\tu 
        &=G_u, &&\tin (0,\tau) \times \Omega,\\
        \del_t\tT-\frac{1}{Q_{\rm,0}}\Delta\tT
        &=G_T, &&\tin (0,\tau) \times \Omega,\\
        \del_t\tq_\rv-\Delta\tq_\rv
        &=G_\rv, &&\tin (0,\tau) \times \Omega,\\
        \del_t\tq_\rc-\Delta\tq_\rc
        &=G_\rc, &&\tin (0,\tau) \times \Omega,\\
        \del_t\tq_\rr-\Delta \tq_\rr
        &=G_\rr, &&\tin (0,\tau) \times \Omega.
		\end{aligned}\right.
\end{equation}
For brevity, we introduce the transformed Laplacian $\cL_1$ as well as the resulting difference $\cL_1 - \Delta$.
In fact, given sufficiently regular $f$, these terms are given by~$\cL_1 f \coloneqq \sum_{j,k,l=1}^3 \frac{\del^2 f}{\del y_k \del y_l} Z_{k,j} Z_{l,j} + \sum_{j,k,l=1}^3 \frac{\del f}{\del y_l} \frac{\del Z_{l,j}}{\del y_k} Z_{k,j}$ and
\begin{equation}\label{eq:transformed Laplacian minus Laplacian}
    \begin{aligned}
        (\cL_1 - \Delta) f 
        &= \sum_{j,k,l=1}^3\frac{\partial^{2}f}{\partial{y_{k}}\partial{y_{l}}}(Z_{k,j}-\delta_{k,j})Z_{l,j} + \sum_{k,l=1}^3\frac{\partial^{2}f}{\partial{y_{k}}\partial{y_{l}}}(Z_{l,k}-\delta_{l,k}) + \sum_{j,k,l=1}^3\frac{\partial f}{\partial{y_{l}}} \frac{\partial Z_{l,j}}{\partial{y_{k}}}Z_{k,j}.
    \end{aligned}
\end{equation}
In the sequel, we denote the transformed term associated with $\nabla \mdiv \tu$ by $\cL_2$.
This leads to
\begin{equation}\label{eq:transformed second comp Lame}
    \begin{aligned}
        \bigl((\cL_2 - \nabla \mdiv) \tu\bigr)_i 
        &= \sum_{j,k,l=1}^3\frac{\partial^{2}\tu_{j}}{\partial{y_{k}}\partial{y_{l}}}(Z_{k,j}-\delta_{k,j})Z_{l,i} + \sum_{j,l=1}^3\frac{\partial^{2}\tu_{j}}{\partial{y_{j}}\partial{y_{l}}}(Z_{l,i}-\delta_{l,i})\\
        &\quad + \sum_{j,k,l=1}^3\frac{\partial\tu_{j}}{\partial{y_{k}}} \frac{\partial Z_{k,j}}{\partial{y_{l}}}Z_{l,i}.
    \end{aligned}
\end{equation}
The transformed terms $G_\rd$, $G_u$, $G_T$, $G_\rv$, $G_\rc$ and $G_\rr$ on the right-hand side of \eqref{eq:Lagrangian coords} are then given by
\begin{equation*}
    G_\rd = -(\trho_\rd-\rho_{\rd,0})\mdiv\tu-\trho_\rd\nabla\tu:[Z^\top-\mathrm{Id}_3],
\end{equation*}
\begin{equation*}
    \begin{aligned}
        (G_{u})_{i}
        &= \Bigl(1-\frac{\trho_\rd \tQ_\rm}{\rho_{\rd,0}Q_{\rm,0}}\Bigr)\del_t\tu_{i} + \frac{\mu}{\rho_{\rd,0}Q_{\rm,0}} (\cL_1 - \Delta) \tu_i + \frac{\mu+\lambda}{\rho_{\rd,0}Q_{\rm,0}} \bigl((\cL_2 - \nabla \mdiv) \tu\bigr)_i\\
        &\quad -\frac{1}{\rho_{\rd,0}Q_{\rm,0}} \cdot \Bigl(\tT(1+\tq_\rv)(Z^\top\nabla \trho_\rd)_{i}+\trho_\rd\tT(Z^\top\nabla\tq_\rv)_{i}\\
        &\qquad+\trho_\rd(1+\tq_\rv)(Z^\top\nabla \tT)_{i} - \trho_\rd\tq_\rr\tV_\rr\sum_{l=1}^3\frac{\del \tu_i}{\del y_l}Z_{l,3} + \trho_\rd\tQ_\rm g \mre_{3}\Bigr),
    \end{aligned}
\end{equation*}
\begin{equation*}
    \begin{aligned}
        G_T
        &= \Bigl(1-\frac{\tQ_\rm}{Q_{\rm,0}}\Bigr)\del_t\tT + \frac{1}{Q_{\rm,0}} \biggl((\cL_1 - \Delta) \tT + \tq_\rr\tV_\rr\sum_{l=1}^3 \frac{\del\tT}{\del y_{l}} Z_{l,3} - (1+\tq_\rv)\tT\nabla \tu:Z^\top\\
        &\qquad + (1+\tT)\Bigl[(\tq_\rv-\tq_\rvs)\tq_\rc+(\tq_\rv-\tq_\rvs)^{+}-\tT\tq_\rr\frac{1+\tq_\rv}{\tQ_\rm}(\tq_\rvs-\tq_\rv)^{+}\Bigr]\biggr),
    \end{aligned}
\end{equation*}
\begin{equation*}
    \begin{aligned}
        G_\rv
        &= (\cL_1 - \Delta) \tq_\rv +\tT\tq_\rr\frac{1+\tq_\rv}{\tQ_\rm}(\tq_\rvs-\tq_\rv)^{+}-(\tq_\rv-\tq_\rvs)\tq_\rc-(\tq_\rv-\tq_\rvs)^{+},\\
        G_\rc
        &= (\cL_1 - \Delta) \tq_\rc +(\tq_\rv-\tq_\rvs)\tq_\rc+(\tq_\rv-\tq_\rvs)^{+}-(\tq_\rc-1)^{+}-\tq_\rc\tq_\rr, \tand\\
        G_\rr
        &= (\cL_1 - \Delta) \tq_\rr +(\tq_\rc-1)^{+}+\tq_\rc\tq_\rr-\tT\tq_\rr\frac{1+\tq_\rv}{\tQ_\rm}(\tq_\rvs-\tq_\rv)^{+}\\
        &\quad +\tV_\rr\sum_{l=1}^3 \frac{\del \tq_\rr}{\del y_{l}}Z_{l,3} + \tq_\rr\sum_{l=1}^3 \frac{\del \tV_\rr}{\del y_{l}} Z_{l,3}
        +\frac{1}{\trho_\rd}\tq_\rr \tV_\rr\sum_{l=1}^3 \frac{\del \trho_\rd}{\del y_{l}} Z_{l,3}.
    \end{aligned}
\end{equation*}
In order to ease notation, we will also write $G = (G_\rd,G_{u},G_T,G_\rv,G_\rc,G_\rr)$ in the sequel.

In view of \eqref{eq:bc fluid} and \eqref{eq:bc heat & moisture}, we also need to transform the boundary conditions.
The periodic boundary conditions on the lateral boundary as well as the Dirichlet boundary conditions are not affected, while for the Neumann boundary conditions, we have to take into account the chain rule.
In fact, using the tilde notation for the respective functions in Lagrangian coordinates, we obtain 
\begin{equation}\label{eq:transformed bc fluid}
    \tu = 0, \ton \Gau \cup \Gal
\end{equation}
for the fluid part as well as 
\begin{equation}\label{eq:transformed bc heat & moisture}
    \begin{aligned}
        \del_{y_3} \tT
        &= (1-Z_{3,3}) \del_{y_3} \tT - Z_{1,3} \del_{y_1} \tT - Z_{2,3} \del_{y_2} \tT \eqqcolon B_{T,1}, &&\ton \Gau,\\
        \del_{y_3} \tT
        &=  (1-Z_{3,3}) \del_{y_3} \tT - Z_{1,3} \del_{y_1} \tT - Z_{2,3} \del_{y_2} \tT \eqqcolon B_{T,2}, &&\ton \Gal,\\
        \del_{y_3} q_j
        &=(1-Z_{3,3}) \del_{y_3} q_j - Z_{1,3} \del_{y_1} q_j - Z_{2,3} \del_{y_2} q_j \eqqcolon B_{q_j ,1}, &&\ton \Gau,\\
        \del_{y_3} q_j
        &= (1-Z_{3,3}) \del_{y_3} q_j - Z_{1,3} \del_{y_1} q_j - Z_{2,3} \del_{y_2} q_j \eqqcolon B_{q_j ,2}, &&\ton \Gal, 
    \end{aligned}
\end{equation}
where $j \in \{\rv,\rc,\rr\}$.
We will also use $B_i = (B_{T,i},B_{\rv,i},B_{\rc,i},B_{\rr,i})$ in the following.
The initial conditions do not change when transforming to Lagrangian coordinates, so we have
\begin{equation}\label{eq:init conds Lagrange}
    \begin{aligned}
        \trho_\rd(0) 
        &= \rho_{\rd,0}, \enspace \tu(0) = u_0, \enspace \tT(0) = T_0,\\
        \tq_\rv(0) 
        &= q_{\rv,0}, \enspace \tq_\rc(0) = q_{\rc,0} \tand \tq_\rr(0) = q_{\rr,0}.
    \end{aligned}
\end{equation}

The aim now is to show that the transformed system of equations in Lagrangian coordinates \eqref{eq:Lagrangian coords}--\eqref{eq:transformed bc heat & moisture} admits a unique solution.
This will be addressed in the following sections.
The fundamental idea is to use a fixed point procedure based on a good understanding of the underlying linearized problem.
The study of the problem is the topic of the following section.

\section{Linear Theory}\label{sec:lin theory}

In this section, we tackle the linearized problem that is associated with the coupled problem in Lagrangian coordinates \eqref{eq:Lagrangian coords}--\eqref{eq:init conds Lagrange}.
The shape of the boundary conditions as revealed in \eqref{eq:transformed bc fluid} and \eqref{eq:transformed bc heat & moisture} calls for an {\em optimal data result}, i.e., we are interested in finding the suitable spaces for the data and establishing estimates of the solution to the linear problem by the data.

As before, we consider the domain $\Omega= G \times (0,1)$ where $G = (0,1)^2$. 
We start by analyzing the linearized equation that corresponds to \eqref{eq:Lagrangian coords}$_1$, namely
\begin{equation}
	\left\{
	\begin{aligned}
        \dt \trho_\rd + \rho_{\rd,0} \mdiv \tu &=f_\rd, &&\tin (0,\tau) \times \Omega, \\
		  \trho_\rd(0)  &= \rho_{\rd,0}  , &&\tin \Omega,  
	\end{aligned}
	\right. 
	\label{eq:conteq3}
\end{equation}
the Lam\'{e} system with homogeneous Dirichlet boundary conditions given by
\begin{equation}
	\left\{
	\begin{aligned}
		\dt \tu - \frac{1}{\rho_{\rd,0} Q_{\rm,0}} \rL \tu &=  f_u, &&\tin (0,\tau) \times \Omega,  \\ 
         \tu 
         &= 0, &&\ton (0,\tau) \times (\Gau \cup \Gal),\\
        \tu(0)  
        &=u_0 ,  &&\tin \Omega,
	\end{aligned}
	\right. 
	\label{eq:Lame1}
\end{equation}
as well as the non-homogeneous heat equation with Neumann boundary conditions
\begin{equation}
	\left\{
	\begin{aligned}
		\dt \tq - \Delta \tq
        &=  f_h, &&\tin (0,\tau) \times \Omega,  \\
        \dz \tq   
        &= g_{h,1}  , &&\ton (0,\tau) \times \Gau, \\ 
		\dz \tq    
        &= g_{h,2}  , &&\ton (0,\tau) \times \Gal, \\ 
        \tq(0)  
        &= q_{h,0},  &&\tin \Omega.
	\end{aligned}
	\right. 
	\label{eq:Heat1}
\end{equation}
As indicated above, all variables are subject to periodic boundary conditions on the lateral boundary. 

The following assumptions are crucial for solving the boundary value problems~\eqref{eq:Lame1} and \eqref{eq:Heat1} subject to inhomogeneous data.

\begin{asu}\label{assu:data}
Given $p$, $q \in (1,\infty)$ such that $\nicefrac{2}{p} + \nicefrac{1}{q} \notin \{1,2\}$ and $\nicefrac{2}{p}+ \nicefrac{3}{q}<1$. Moreover, suppose that $f_\rd$, $f_u$, $f_h$, $b_1$, $b_2$, $g_{h,1}$, $g_{h,2}$, $\rho_{\rd,0}$, $u_0$ and $q_{h,0}$, satisfy
\begin{enumerate}[(i)]
    \item $(\rho_{\rd,0}, u_0, q_{h,0}) \in \rH_\per^{1,q}(\Omega) \times \rB_{q,p,\per}^{2(1-\nicefrac{1}{p})}(\Omega)^3 \times \rB_{q,p,\per}^{2(1-\nicefrac{1}{p})}(\Omega)$, 
    and $\rho_{\rd,0}$, $q_{h,0}$ satisfy $M_1 \leq \rho_{\rd,0}, q_{h,0} \leq M_2$ in $\Omega$ for  some constants $M_1, M_2 >0$, and we have $u_0 = 0$ on $\Gau \cup \Gal$ provided $\nicefrac{2}{p} + \nicefrac{1}{q} < 2$,
    \item $(f_\rd,f_u,f_h) \in \rLp(0,\tau;\rH_\per^{1,q}(\Omega)) \times  \rLp(0,\tau;\rLq(\Omega)^3) \times \rLp(0,\tau;\rLq(\Omega)) \eqqcolon \E_{0,\rd,\tau} \times \E_{0,u,\tau} \times \E_{0,h,\tau}$,
    \item $g_{h,i} \in \rF_{p,q}^{\nicefrac{1}{2}-\nicefrac{1}{2q}}(0,\tau;\rLq(G)) \cap \rLp(0,\tau;\rB_{q,q, \per}^{1 - \nicefrac{1}{q}}(G))=:\F_{h,\tau},\ $  and \ $\dz q_{h,0} = g_{h,i}(0)$ if $\nicefrac{2}{p} + \nicefrac{1}{q} < 1$, for $i\in \{ 1,2\}$.
\end{enumerate}
\end{asu}

Under these assumptions, we have
$\rB^{2\left(1-\nicefrac{1}{p}\right)}_{q,p,\per}(\Omega) \hookrightarrow \rH^{1,q}_\per(\Omega)$ and $\rho_{\rd,0} Q_{m,0} > 0$.
In particular, $\rho_{\rd,0} Q_{m,0} \in \rH^{1,q}_\per(\Omega)$ by virtue of the conditions on $p,q$ and $\rH^{1,q}$ behaving like a Banach algebra with respect to multiplication. Sobolev embedding further implies that
$\rho_{\rd,0} Q_{m,0} \in \mathrm{BUC}^\gamma(\bar{\Omega})$.
Hence, by \cite{DHP:07}, optimal data results concerning \eqref{eq:Lame1} and \eqref{eq:Heat1} are known for the half-space and can be adapted to the layer $\Omega_L = \R^2 \times (0,1)$ via the same localization procedure.
Finally, these results carry over to the domain
$\Omega = G \times (0,1)$
by extending periodic functions on $G$ to the full space $\R^2$, as demonstrated on p.~1082 in \cite{HK:16}.

\begin{lem}\label{lem:lin opt data}

Let $\tau>0$, and assume $p,q$ and that the data $(f_\rd,f_u,f_h,g_{h,1},g_{h,2},\rho_{\rd,0},u_0,q_{h,0})$ satisfy \autoref{assu:data}.

\begin{enumerate}
    \item Eq.~\eqref{eq:Lame1} has a unique solution
    \begin{equation*}
        \tu \in \rH^{1,p}(0,\tau;\rLq(\Omega)^3) \cap \rLp(0,\tau;\rH_\per^{2,q}(\Omega)^3) \eqqcolon \E_{1,u,\tau},
    \end{equation*}
    satisfying $\| \tu \|_{\E_{1,u,\tau}} \le C \bigl(\| f_u\|_{\E_{0,u,\tau}}  + \| u_0 \|_{\rB_{q,p}^{2(1-\nicefrac{1}{p})}(\Omega)}\bigr)$ for some constant $C > 0$.
    \item Let $\tu \in \E_{1,u,\tau}$ be the unique solution to \eqref{eq:Lame1} resulting from~(i).
    Then \eqref{eq:conteq3} has a unique solution 
    $\trho_\rd \in \rH^{1,p}(0,\tau; \rH_\per^{1,q}(\Omega)) \eqqcolon \E_{1,\rd,\tau}$ such that the estimate 
    \begin{equation*}
        \| \trho_\rd \|_{\E_{1,\rd,\tau}} \le C_1\bigl(\| f_\rd \|_{\E_{0,\rd,\tau}} + \| \tu \|_{\E_{1,u,\tau}} + \| \rho_{\rd,0} \|_{\rH^{1,q}(\Omega)} \bigr )
    \end{equation*}
    holds for some $C_1 > 0$.
    \item The linear problem \eqref{eq:Heat1} has a unique solution 
    \begin{equation*}
        \tq \in \rH^{1,p}(0,\tau;\rLq(\Omega)) \cap \rLp(0,\tau;\rH_\per^{2,q}(\Omega)) \eqqcolon \E_{1,h,\tau},
    \end{equation*}
    with $\| \tq \|_{\E_{1,h,\tau}} \le C\bigl(\| f_h \|_{\E_{0,h,\tau}} + \| g_{h,1} \|_{ \F_{h,\tau}} +\| g_{h,2} \|_{ \F_{h,\tau}} + \| q_{h,0} \|_{\rB_{q,p}^{2(1-\nicefrac{1}{p})}(\Omega) }\bigr)$ for a constant $C > 0$.
\end{enumerate}
\end{lem}

Let us emphasize that the constant $C > 0$ in the estimates of the respective solutions by the data in \autoref{lem:lin opt data} can be chosen independent of $\tau$ provided we consider homogeneous initial data, i.e., $u_0 = q_{h,0}=0$. 

The above preparations pave the way for the main result of this section on the solvability of the linearized problem that corresponds to \eqref{eq:Lagrangian coords}--\eqref{eq:init conds Lagrange}.
In fact, for~$j \in \{\rv,\rc,\rr\}$, the linearized problem under consideration reads as 
\begin{equation}\label{eq:complete lin syst in Lagrangian coords}
    \left\{
    \begin{aligned}
        \dt \trho_\rd + \rho_{\rd,0} \mdiv \tu 
        &=f_\rd, &&\tin (0,\tau) \times \Omega,\\
        \dt \tu - \frac{1}{\rho_{\rd,0} Q_{\rm,0}} \rL \tu 
        &=  f_u, &&\tin (0,\tau) \times \Omega,\\
        \dt \tT - \frac{1}{Q_{\rm,0}} \Delta \tT
        &= f_T, &&\tin (0,\tau) \times \Omega,\\
        \dt \tq_\rv - \Delta \tq_\rv
        &= f_\rv, &&\tin (0,\tau) \times \Omega,\\
        \dt \tq_\rc - \Delta \tq_\rc
        &= f_\rc, &&\tin (0,\tau) \times \Omega,\\
        \dt \tq_\rr - \Delta \tq_\rr
        &= f_\rr, &&\tin (0,\tau) \times \Omega,\\
        \tu = 0, \enspace \dz \tT = g_{T,1}, \enspace \dz \tq_j
        &= g_{j,1}, &&\ton (0,\tau) \times \Gau,\\
        \tu = 0, \enspace \dz \tT = g_{T,2}, \enspace \dz \tq_j
        &= g_{j,2}, &&\ton (0,\tau) \times \Gal,\\
        \trho(0) = \rho_{\rd,0}, \enspace \tu(0) = u_0, \enspace \tT(0) = T_0, \enspace \tq_j(0)
        &= q_{j,0}, &&\tin \Omega.
    \end{aligned}
    \right.
\end{equation}

The following proposition on the solvability of \eqref{eq:complete lin syst in Lagrangian coords} and an estimate of the resulting solution by the data is a consequence of \autoref{lem:lin opt data}, where for $\tu$ and $\trho$, we employ~(i) and~(ii), respectively, while for $\tT$ as well as $\tq_j$, $j \in \{\rv,\rc,\rr\}$, we make use of part~(iii).

\begin{prop}\label{prop:opt data result moisture compre NSE}
Consider $\tau > 0$ as well as $\nicefrac{2}{p}+ \nicefrac{3}{q}<1$, and suppose that 
\begin{enumerate}[(i)]
    \item $f_\rd \in \E_{0,\rd}$, $f_u \in \E_{0,u}$, $f_T$, $f_\rv$, $f_\rc$, $f_\rr \in \E_{0,h}$,
    \item $g_{T,i}$, $g_{j,i} \in \F_h$, for $i=1,2$ and $j \in \{\rv,\rc,\rr\}$, as well as
    \item $\rho_{\rd,0} \in \rH_\per^{1,q}(\Omega)$, $u_0 \in \rB_{q,p,\per}^{2(1-\nicefrac{1}{p})}(\Omega)^3$, $T_0$, $q_{j,0} \in \rB_{q,p,\per}^{2(1-\nicefrac{1}{p})}(\Omega)$, for $j \in \{\rv,\rc,\rr\}$, also satisfying $u_0 = 0$, $\dz T_0 = g_{T,i}(0)$ and $\dz q_{j,0} = g_{j,i}(0)$, for $i=1,2$, $j \in \{\rv,\rc,\rr\}$, on the respective boundaries.
\end{enumerate}
Then there is a unique solution $\trho_\rd \in \E_{1,\rd,\tau}$, $\tu \in \E_{1,u,\tau}$, $\tT$, $\tq_j \in \E_{1,h,\tau}$, $j \in \{\rv,\rc,\rr\}$, to \eqref{eq:complete lin syst in Lagrangian coords}.
Furthermore, for some constant $C > 0$, this solution admits the estimate
\begin{equation}\label{eq:opt data est}
    \begin{aligned}
        &\quad \| \trho_\rd \|_{\E_{1,\rd,\tau}} + \| \tu \|_{\E_{1,u,\tau}} + \| \tT \|_{\E_{1,h,\tau}} + \| \tq_\rv \|_{\E_{1,h,\tau}} + \| \tq_\rc \|_{\E_{1,h,\tau}} + \| \tq_\rr \|_{\E_{1,h,\tau}}\\
        &\le C\Bigl(\| f_\rd \|_{\E_{0,\rd,\tau}} + \| f_u \|_{\E_{0,u,\tau}} + \| f_T \|_{\E_{0,h,\tau}} + \| f_\rv \|_{\E_{0,h,\tau}} + \| f_\rc \|_{\E_{0,h,\tau}}\\
        &\quad + \| f_\rr \|_{\E_{0,h,\tau}} + \sum_{i=1}^2 \bigl ( \| g_{T,i} \|_{\F_{h,\tau}}+ \| g_{\rv,i} \|_{\F_{h,\tau}} + \| g_{\rc,i} \|_{\F_{h,\tau}} + \| g_{\rr,i} \|_{\F_{h,\tau}}\bigr)\\
        &\quad + \| \rho_{\rd,0} \|_{\rH^{1,q}(\Omega)} + \| (u_0,T_0,q_{\rv,0},q_{\rc,0},q_{\rr,0}) \|_{\rB_{q,p}^{2(1-\nicefrac{1}{p})}(\Omega)}\Bigr).
    \end{aligned}
\end{equation}
\end{prop}

Again, we remark that the constant $C$ in the optimal data estimate \eqref{eq:opt data est} can be chosen independent of the time $\tau$ if one considers homogeneous initial data.

\section{Local Strong Well-Posedness}\label{sec: local}

In this section, we prove the first main result of this paper, namely \autoref{thm:localWP} on the local strong well-posedness.
The section is further divided into two subsections, namely one on the nonlinear estimates, and one on the actual proof of \autoref{thm:localWP}.

\subsection{Nonlinear estimates}\label{ssec:nonlin ests}
\ 

This subsection is dedicated to estimating the nonlinear terms $G_u$, $G_T$ and $G_j$, $j \in \{\rv,\rc,\rr\}$ as made precise after \eqref{eq:Lagrangian coords}.
To this end, we first collect some auxiliary estimates.
In the following, we denote by~$\tau \in (0,\infty)$ a strictly positive, finite time.

The lemma below allows for a uniform-in-time estimate of a function in its trace space norm by the initial value and the maximal regularity norm, where the constant in the estimate is {\em independent} of $\tau > 0$.
Such an estimate is desirable for the fixed point argument, as it helps to choose time sufficiently small to get the self-map and contraction property.
For a proof in the situation of a smooth and bounded domain, which also carries over to the present situation of a unit box with horizontal periodicity, we refer e.g., to \cite[Thm.~2.4]{Shi:15}.

\begin{lem}\label{lem:est in Linfty Besov}
Consider $p$, $q \in (1,\infty)$, $\tau > 0$ and
\begin{equation*}
    f \in \E_{1,h,\tau} = \rH^{1,p}(0,\tau;\rLq(\Omega)) \cap \rLp(0,\tau;\rH_\per^{2,q}(\Omega)).
\end{equation*}
Then there is a constant $C$ which can be chosen independent of $\tau$ such that
\begin{equation*}
    \sup_{t \in (0,\tau)} \| f(t) \|_{\rB_{q,p}^{2(1-\nicefrac{1}{p})}(\Omega)} \le C\bigl(\| f(0) \|_{\rB_{q,p}^{2(1-\nicefrac{1}{p})}(\Omega)} + \| f \|_{\E_{1,h,\tau}}\bigr).
\end{equation*}
\end{lem}

We remark that the estimate in \autoref{lem:est in Linfty Besov} is also valid in the vector-valued situation, i.e., for the fluid velocity $u$, and with $\E_{1,h,\tau}$ replaced by $\E_{1,u,\tau}$ accordingly.

With regard to \autoref{lem:lin opt data} and \autoref{prop:opt data result moisture compre NSE}, the following lemma on estimates of functions in the Triebel-Lizorkin space proves useful.
Again, let us observe that the result on domains with smooth boundaries as e.g., shown in \cite[Prop.~6.4]{DHP:07} transfers to the present setting of a unit box with horizontal periodicity upon using the procedure sketched before \autoref{lem:lin opt data}.

\begin{lem}\label{lem:est of normal derivative}
Let $p$, $q \in (1,\infty)$ and $\tau > 0$, and consider
\begin{equation*}
    f \in \E_{1,h,\tau} = \rH^{1,p}(0,\tau;\rLq(\Omega)) \cap \rLp(0,\tau;\rH_\per^{2,q}(\Omega)).
\end{equation*}
Then it follows that $\left.\dz f\right|_G \in \F_{h,\tau} = \rF_{p,q}^{\nicefrac{1}{2}-\nicefrac{1}{2q}}(0,\tau;\rLq(G)) \cap \rLp(0,\tau;\rB_{q,q, \per}^{1 - \nicefrac{1}{q}}(G))$, and we obtain the estimate
\begin{equation*}
    \| \dz f \|_{\F_{h,\tau}} \le C \bigl(\| f(0) \|_{\rB_{q,p}^{2(1-\nicefrac{1}{p})}(\Omega)} + \| f \|_{\E_{1,h,\tau}}\bigr)
\end{equation*}
for some constant $C > 0$ which can be chosen independent of $\tau$.
\end{lem}

A proof of the lemma below on an estimate of a bilinear boundary term can be found in \cite[Prop.~2.7]{HMTT:19}.

\begin{lem}\label{lem:est of bilin term on boundary}
Consider $p \in (2,\infty)$, $q \in (3,\infty)$ and $s \in (0,1)$ with $s + \nicefrac{1}{p} < 1$, let $\tau > 0$, and take into account Banach spaces $U_1$, $U_2$, $U_3$ as well as a bounded bilinear map $\Phi \colon U_1 \times U_2 \to U_3$.
Moreover, let~$f \in \rF_{p,q}^s(0,\tau;U_1)$ and $g \in \rH^{1,p}(0,\tau;U_2)$ with $g(0) = 0$.
Then there exist $\delta = \delta(p,q,s) > 0$ and a constant $C > 0$, which can be chosen independent of $\tau$, such that
\begin{equation*}
    \| \Phi(f,g) \|_{\rF_{p,q}^s(0,\tau;U_3)} \le C \tau^{\delta} \cdot \| f \|_{\rF_{p,q}^s(0,\tau;U_1)} \cdot \| g \|_{\rH^{1,p}(0,\tau;U_2)}.
\end{equation*}
\end{lem}

Next, we collect estimates of the unknowns.
For this purpose, let $0<\tau< \infty$. 
We define the full solution space
\begin{equation}\label{eq:full sol space}
    \E_{1,\tau} \coloneqq \E_{1,\rd,\tau} \times \E_{1,u,\tau} \times \E_{1,T,\tau} \times \E_{1,\rv,\tau} \times \E_{1,\rc,\tau} \times \E_{1,\rr,\tau}.
\end{equation}

For brevity, we also introduce dedicated pieces of notation for the full data space and space for the boundary data, namely $\E_{0,\tau} \coloneqq \E_{0,\rd,\tau} \times \E_{0,u,\tau} \times (\E_{0,h,\tau})^4$ and $\F_{\tau} \coloneqq  (\F_{h,\tau})^4$.
For later reference, we will also use 
\begin{equation*}
    \rX_\gamma \coloneqq \rH_\per^{1,q}(\Omega) \times \rB_{q,p,\per}^{2(1-\nicefrac{1}{p})}(\Omega)^3 \times \rB_{q,p,\per}^{2(1-\nicefrac{1}{p})}(\Omega)^4
\end{equation*}
for the trace space, which is the suitable space for the initial data.

The auxiliary lemma on estimates of the variables then reads as follows.
The assertions of~(i) can be deduced from using H\"older's inequality in time.
The first estimate in (ii) is implied by \autoref{lem:est in Linfty Besov} together with the Sobolev embedding $\rB_{q,p}^{2(1-\nicefrac{1}{p})}(\Omega) \hookrightarrow \rH^{1,q}(\Omega)$ thanks to $p > 2$.
The second and third estimate in~(ii) follow from the first one together with complex interpolation and the embedding $\rH^{s,q}(\Omega) \hookrightarrow \rL^\infty(\Omega)$ for $s > \nicefrac{3}{q}$, respectively.

\begin{lem}\label{lem:ests of vars}
Consider $p \in (2,\infty)$, $q \in (3,\infty)$, and let $(\trho_\rd,\tu,\tT,\tq_\rv,\tq_\rc,\tq_\rr) \in \E_{1,\tau}$ be a solution to \eqref{eq:complete lin syst in Lagrangian coords} with data as made precise in \autoref{prop:opt data result moisture compre NSE}.
Then there is $C > 0$, independent of $\tau$, such that
\begin{enumerate}[(i)]
    \item the density $\trho_\rd$ satisfies $\| \trho_\rd -\rho_{\rd,0} \|_{\rL_\tau^\infty(\rH_x^{1,q})} \le C \tau ^{\nicefrac{1}{p^{'}}}$, $\| \trho_\rd \|_{\rL_\tau^\infty(\rH_x^{1,q})} \le C$ and $\| \trho_\rd \|_{\rL_\tau^p(\rH_x^{1,q})} \le C \tau^{\nicefrac{1}{p}}$,
    \item the fluid velocity $\tu$ fulfills $\| \tu \|_{\rL_\tau^\infty(\rH_x^{1,q})} \le C$ as well as $\| \tu \|_{\rL_\tau^p(\rH_x^{1+s,q})} \le C \tau^{\nicefrac{(1-s)}{2p}}$ for $s \in (0,1)$, and $\| \tu \|_{\rL_\tau^p(\rL_x^\infty)} + \| \nabla \tu \|_{\rL_\tau^p(\rL_x^\infty)} \le C \tau^{\nicefrac{(1-s)}{2p}}$ for $s \in (\nicefrac{3}{q},1)$,
    \item for the temperature $\tT$, it holds that $\| \tT \|_{\rL_\tau^\infty(\rH_x^{1,q})} \le C$ and
    \begin{equation*}
        \begin{aligned}
            \| \tT \|_{\rL_\tau^p(\rH_x^{1+s,q})} 
            &\le C \tau^{\nicefrac{(1-s)}{2p}} \tfor s \in (0,1), \tand\\
            \| \nabla \tT \|_{\rL_\tau^p(\rL_x^\infty)} 
            &\le C \tau^{\nicefrac{(1-s)}{2p}} \tfor s \in (\nicefrac{3}{q},1), \tand
        \end{aligned}
    \end{equation*}    
    \item the estimates for $\tT$ as stated in~(iii) remain valid for the mixing ratios $\tq_\rv$, $\tq_\rc$ and $\tq_\rr$.
\end{enumerate}
\end{lem}

We complete the preparation of the nonlinear estimates with the subsequent lemma on estimates of the coordinate transform to Lagrangian coordinates.
First, we deduce from \eqref{eq:CP Lagrange} and the change of coordinates as specified in \autoref{sec:Lagrangian coords} that the diffeomorphism $X$ takes the shape
\begin{equation}\label{eq:shape of X}
    X(t,y) = y + \int_0^t \tu(s,y) \srd s.
\end{equation}
Furthermore, let us recall the pieces of notation $Y(t,\cdot) = [X(t,\cdot)]^{-1}$ as well as~$Z(t,y) = Z_{i,j} = [\nabla X]^{-1}(t,y)$.
In the following lemma, we will also justify the existence of $Z$ for sufficiently small time.
The proof of the lemma follows by standard arguments, and we refer for instance to \cite[Sec.~2.3]{HMTT:19} for further details.

\begin{lem}\label{lem:ests of trafo}
Let $p \in (2,\infty)$, $q \in (3,\infty)$, and recall that $X$ is associated with $\tu \in \E_{1,u,\tau}$.
\begin{enumerate}[(i)]
    \item There is $C > 0$, independent of $\tau$, such that
    \begin{equation*}
        \sup_{t \in (0,\tau)} \| \nabla X(t,\cdot) - \mathrm{Id}_3 \|_{\rL_x^\infty} \le C \cdot \sup_{t \in (0,\tau)} \| \nabla X(t,\cdot) - \mathrm{Id}_3 \|_{\rH_x^{1,q}} \le C \tau^{\nicefrac{1}{p'}}.
    \end{equation*}
    Thus, there is $\Tilde{\tau}$ such that for $0 < t < \tau \le \Tilde{\tau}$, we have $\| \nabla X(t,\cdot) - \mathrm{Id}_3 \|_{\rL_x^\infty} \le \nicefrac{1}{2}$, so~$\nabla X(t,\cdot)$ is invertible, and $Z = [\nabla X]^{-1}$ is thus well-defined.
    \item For $0 < \tau \le \Tilde{\tau}$, with $\Tilde{\tau}$ as made precise in~(i), there exists $C > 0$, independent of $\tau$, so that
    \begin{equation*}
        \begin{aligned}
            \| Z \|_{\rH_\tau^{1,p}(\rH_x^{1,q})} + \| Z \|_{\rL_\tau^\infty(\rH_x^{1,q})} + \| \nabla X \|_{\rH_\tau^{1,p}(\rH_x^{1,q})} + \| \nabla X \|_{\rL_\tau^\infty(\rH_x^{1,q})} \le C.
        \end{aligned}
    \end{equation*}
    \item Similar estimates as in (i) are valid for $Z$ and $Z^\top$, i.e., for some $\tau$-independent constant $C > 0$, we have
    \begin{equation*}
        \sup_{t \in (0,\tau)} \| Z(t,\cdot) - \mathrm{Id}_3 \|_{\rH_x^{1,q}} + \sup_{t \in (0,\tau)} \| Z^\top(t,\cdot) - \mathrm{Id}_3 \|_{\rH_x^{1,q}} \le C \tau^{\nicefrac{1}{p'}}.
    \end{equation*}
    \item For all $j$, $k$, $l = 1,2,3$, there is $C > 0$, independent of $\tau$, such that~$\bigl\| \frac{\del Z_{l,j}}{\del y_k} \bigr\|_{\rL_\tau^\infty(\rL_x^q)} \le C \tau^{\nicefrac{1}{p'}}$.
\end{enumerate}
\end{lem}

In the following, we estimate the nonlinear terms, which result from changing to Lagrangian coordinates, and which have been calculated in \autoref{sec:Lagrangian coords}, by a constant that is shrinking in $\tau$.

\begin{lem}\label{lem: estimate non lin}
Let $p \in (2,\infty)$, $q \in (3,\infty)$, and consider a solution
\begin{equation*}
    \tz_1 \coloneqq (\trho_{\rd,1},\tu_1,\tT_1,\tq_{\rv,1},\tq_{\rc,1},\tq_{\rr,1}) \in \E_{1,\tau},
\end{equation*}
with $\tau \le \Tilde{\tau}$, to the linearized problem \eqref{eq:complete lin syst in Lagrangian coords} with data $(f_\rd,f_u,f_T,f_j)$, $(g_{T,i},q_{j,i})$, $(\rho_{\rd,0},u_0,T_0,q_{j,0})$, for $j \in \{\rv,\rc,\rr\}$, $i=1,2$, as made precise in \autoref{prop:opt data result moisture compre NSE}.

Then there exists a constant $C(\tau)$, which only depends on $p$, $q$ as well as $\Tilde{\tau}$ and satisfies $C(\tau) \to 0$ as~$\tau \to 0$, such that for the terms $G(\tz_1)$ and $B_i(\tz_1)$ introduced in \autoref{sec:Lagrangian coords}, we have
\begin{equation*}
    \| G(\tz_1) \|_{\E_{0,\tau}} + \sum_{i=1}^2 \| B_i(\tz_1) \|_{\F_\tau} \le C(\tau).
\end{equation*}
\end{lem}

\begin{proof}
We successively estimate the nonlinear terms $G = (G_\rd,G_{u},G_T,G_\rv,G_\rc,G_\rr)$.
Many terms are of a similar shape, so we only focus on selected terms.

Let us start with the term $G_\rd$.
From the Banach algebra structure of $\rH^{1,q}(\Omega)$ thanks to $q > 3$, \autoref{lem:ests of vars}(i) for the estimate of $\trho_{\rd,1}-\rho_{\rd,0}$ and $\tz_1 \in \E_{1,\tau}$, it first follows that
\begin{equation*}
    \| (\trho_{\rd,1}-\rho_{\rd,0})\mdiv\tu_1 \|_{\rL_\tau^p(\rH_x^{1,q})} \le \| \trho_{\rd,1}-\rho_{\rd,0} \|_{\rL_\tau^\infty(\rH_x^{1,q})} \cdot \| \tu_1 \|_{\rL_\tau^p(\rH_x^{2,q})} \le C \tau^{\nicefrac{1}{p'}}.
\end{equation*}
In the same way, this time employing \autoref{lem:ests of vars}(i) for the boundedness of $\trho_{\rd,1}$, and using \autoref{lem:ests of trafo}(iii) for $Z^\top-\mathrm{Id}_3$, we get
\begin{equation*}
    \begin{aligned}
        \| \trho_{\rd,1}\nabla\tu_1:[Z^\top-\mathrm{Id}_3] \|_{\rL_\tau^p(\rH_x^{1,q})} 
        &\le \| \trho_{\rd,1} \|_{\rL_\tau^\infty(\rH_x^{1,q})} \cdot \| \tu_1 \|_{\rL_\tau^p(\rH_x^{2,q})} \cdot \| Z^\top - \mathrm{Id}_3 \|_{\rL_\tau^\infty(\rH_x^{1,q})}\\
        &\le C \tau^{\nicefrac{1}{p'}}.
    \end{aligned}
\end{equation*}
This shows that $G_\rd$ admits estimates of the desired shape.

The next nonlinear term under consideration is $G_u$.
With regard to the first term of $G_u$, we make use of the boundedness of $\rho_{\rd,0}$ and $q_{j,0}$, $j \in \{\rv,\rc,\rr\}$, from below, implying that $\nicefrac{1}{\rho_{\rd,0} Q_{\rm,0}}$ to derive that
\begin{equation}\label{eq:est of first term of G_u}
    \Bigl\| (1-\frac{\trho_{\rd ,1}\tQ_{\rm,1}}{\rho_{\rd,0}Q_{\rm,0}}\Bigr)\del_t\tu_1 \Bigr\|_{\rL_\tau^p(\rL_x^p)} \le \| \rho_{\rd,0} Q_{\rm,0} - \trho_{\rd,1} \tQ_{\rm,1} \|_{\rL_\tau^\infty(\rL_x^\infty)} \cdot \| \tu_1 \|_{\rH_\tau^{1,p}(\rL_x^q)}.
\end{equation}
The second factor can be controlled by a constant thanks to $\tz_1 \in \E_{1,\tau}$, so it remains to deal with the first one.
For this purpose, we expand it as
\begin{equation}\label{eq:est of first factor first term G_u}
    \begin{aligned}
        \| \rho_{\rd,0} Q_{\rm,0} - \trho_{\rd,1} \tQ_{\rm,1} \|_{\rL_\tau^\infty(\rL_x^\infty)} 
        \le \| (\rho_{\rd,0} - \trho_{\rd,1}) Q_{\rm,0}\|_{\rL_\tau^\infty(\rL_x^\infty)}  + \| \trho_{\rd,1} (Q_{\rm,0} - \tQ_{\rm,1}) \|_{\rL_\tau^\infty(\rL_x^\infty)}.
    \end{aligned}
\end{equation}
By \autoref{lem:ests of vars}(i) as well as the embedding $\rB_{q,p}^{2-\nicefrac{2}{p}}(\Omega) \hookrightarrow \rL^\infty(\Omega)$ thanks to the assumptions on $p$ and~$q$, which yields that $Q_{\rm,0}$ is bounded, we find that the first addend can be bounded by a positive power of $\tau$.
On the other hand, for the second added, we exploit \autoref{lem:ests of vars}(i) for the boundedness of~$\| \trho_{\rd,1} \|_{\rL_\tau^\infty(\rL_x^\infty)}$.
In order to control the remaining factor $\| Q_{\rm,0} - \tQ_{\rm,1} \|_{\rL_\tau^\infty(\rL_x^\infty)}$, we invoke a so-called {\em reference solution}.
More precisely, we first observe that $Q_{\rm,0} - \tQ_{\rm,1} = \sum_{j \in \{\rv,\rc,\rr\}} (q_{j,0} - \tq_{j,1})$.
A reference solution $\tq_j^*$ then is a solution to the heat equation subject with zero right-hand side and subject to homogeneous Neumann boundary conditions, but with initial data $q_{j,0}$, i.e., $\tq_j^*$ solves
\begin{equation*}
    \left\{
    \begin{aligned}
        \dt \tq_j^* - \Delta \tq_j^*
        &= 0, &&\tin (0,\tau) \times \Omega,\\
        \dz \tq_j^*
        &= 0, &&\ton (0,\tau) \times (\Gau \cup \Gal),\\
        \tq_j^*(0)
        &= q_{j,0}, &&\tin \Omega.
    \end{aligned}
    \right.
\end{equation*}
With the notation $\rY_\gamma^j \coloneqq \rB_{q,p}^{2-\nicefrac{2}{p}}(\Omega)$, indicating that, up to boundary conditions, which do not affect the norm, $\rB_{q,p}^{2-\nicefrac{2}{p}}(\Omega)$ is the time trace space associated with the variable $q_j$, it then follows that
\begin{equation}\label{eq:diff init value and ref sol}
    \| q_{j,0} - \tq_j^* \|_{\rL_\tau^\infty(\rL_x^\infty)} \le C \cdot \| q_{j,0} - \tq_j^* \|_{\rL_\tau^\infty(\rY_\gamma^j)} \to 0, \tas \tau \to 0.
\end{equation}
On the other hand, using of the embedding $\prescript{}{0}{\E_{1,h,\tau}} \hookrightarrow \rBUC([0,\tau];\rY_\gamma^j)$, see e.g., \cite[Thm.~III.4.10.2]{Ama:95}, the observation that the embedding constant can be chosen independent from $\tau$ thanks to the homogeneous initial values, and the fact that $\E_{1,h,\tau}$ is an integral norm, we find that
\begin{equation}\label{eq:diff ref sol and sol}
    \| \tq_j^* - \tq_{j,1} \|_{\rL_\tau^\infty(\rL_x^\infty)} \le C \cdot \| \tq_j^* - \tq_{j,1} \|_{\rL_\tau^\infty(\rY_\gamma^j)} \le C \cdot \| \tq_j^* - \tq_{j,1} \|_{\E_{1,h,\tau}} \to 0, \tas \tau \to 0.
\end{equation}
Making use of \eqref{eq:diff init value and ref sol} and \eqref{eq:diff ref sol and sol}, and invoking the above arguments, we find that \eqref{eq:est of first term of G_u} can be estimated in the desired way.

Now we handle the term $\nicefrac{\mu}{(\rho_{\rd,0}Q_{\rm,0})}\sum_{j,k,l=1}^3\frac{\partial^{2}(\tu_1)_i}{\partial{y_{k}}\partial{y_{l}}}(Z_{k,j}-\delta_{k,j})Z_{l,j}$.
The boundedness of $\rho_{\rd,0}$ and $q_{j,0}$, $j \in \{\rv,\rc,\rr\}$, from below together with \autoref{lem:ests of trafo}(iii) and~(ii) leads to
\begin{equation*}
    \begin{aligned}
        \biggl\| \frac{\mu}{\rho_{\rd,0}Q_{\rm,0}}\sum_{j,k,l=1}^3\frac{\partial^{2}(\tu_1)_i}{\partial{y_{k}}\partial{y_{l}}}(Z_{k,j}-\delta_{k,j})Z_{l,j} \biggr\|_{\rL_\tau^p(\rL_x^q)}
        &\le C \cdot \| \tu_1 \|_{\rL_\tau^p(\rH_x^{2,q})} \cdot \| Z - \mathrm{Id}_3 \|_{\rL_\tau^\infty(\rL_x^\infty)} \cdot \| Z \|_{\rL_\tau^\infty(\rL_x^\infty)}\\
        &\le C \cdot \| Z - \mathrm{Id}_3 \|_{\rL_\tau^\infty(\rH_x^{1,q})} \cdot \| Z \|_{\rL_\tau^\infty(\rH_x^{1,q})} \le C \tau^{\nicefrac{1}{p'}}.
    \end{aligned}
\end{equation*}
Let us observe that the terms of $G_u$ which involve differences of $Z$ and $\mathrm{Id}_3$ can be handled in a completely similar manner.
Thus, we next focus on the term~$\nicefrac{\mu}{(\rho_{\rd,0}Q_{\rm,0})} \sum_{j,k,l=1}^3\frac{\partial(\tu_1)_{i}}{\partial{y_{l}}} \frac{\partial Z_{l,j}}{\partial{y_{k}}}Z_{k,j}$.
The factor $\nicefrac{\mu}{(\rho_{\rd,0}Q_{\rm,0})}$ can be handled as above, while for $\tu$, we employ \autoref{lem:ests of vars}(ii), for the derivative of $Z$, we make use of \autoref{lem:ests of trafo}(iv), and we handle $Z_{k,j}$ with \autoref{lem:ests of trafo}(ii).
Hence, we infer that
\begin{equation*}
    \begin{aligned}
        \biggl\| \frac{\mu}{(\rho_{\rd,0}Q_{\rm,0})} \sum_{j,k,l=1}^3\frac{\partial(\tu_1)_{i}}{\partial{y_{l}}} \frac{\partial Z_{l,j}}{\partial{y_{k}}}Z_{k,j} \biggr\|_{\rL_\tau^p(\rL_x^q)}
        &\le C \cdot \| \nabla \tu_1 \|_{\rL_\tau^p(\rL_x^\infty)} \cdot \Bigl\| \frac{\del Z_{l,j}}{\del y_k} \Bigr\|_{\rL_\tau^\infty(\rL_x^q)} \cdot \| Z \|_{\rL_\tau^\infty(\rL_x^\infty)}\\
        &\le C \tau^{\nicefrac{1}{p'}} \cdot \| Z \|_{\rL_\tau^\infty(\rH_x^{1,q})} \le C \tau^{\nicefrac{1}{p'}}.
    \end{aligned}
\end{equation*}
We observe that the term $\nicefrac{(\mu+\lambda)}{(\rho_{\rd,0}Q_{\rm,0})}\sum_{j,k,l=1}^3\frac{\partial\tu_{j}}{\partial{y_{k}}} \frac{\partial Z_{k,j}}{\partial{y_{l}}}Z_{l,i}$ can be handled in the exact same way.

The next term to be estimated is $\frac{1}{\rho_{\rd,0}Q_{\rm,0}} \tT_1(1+\tq_{\rv,1})(Z^\top\nabla \trho_{\rd,1})_{i}$ in $G_u$.
The factor $\nicefrac{1}{(\rho_{\rd,0}Q_{\rm,0})}$ can be handled as above, so we do not include it in the below considerations.
On the other hand, for $\tT$ we employ \autoref{lem:ests of vars}(iii), we deal with~$\tq_{v}$  by \autoref{lem:ests of vars}(iv), \autoref{lem:ests of trafo}(iii) is used for the estimate of $Z^\top$,  and $\trho_{\rd} $ can be handled as 
\eqref{eq:est of first factor first term G_u}.
 Hence, we infer that 
\begin{equation*}
    \begin{aligned}
        \| \tT_1(1+\tq_{\rv,1})(Z^\top\nabla \trho_{\rd,1})_{i}\|_{\rL_\tau^p(\rL_x^{q})}
        &\leq C\cdot\|\tT_{1}\|_{\rL_\tau^p(\rL_x^\infty)}\cdot\|1+\tq_{\rv,1}\|_{\rL_\tau^\infty(\rL_x^\infty)} \cdot\| Z^\top\|_{\rL_\tau^\infty(\rL_x^\infty)}\cdot\|\trho_{d,1}\|_{\rL_\tau^\infty(\rH_x^{1,q})}\\
        &\leq C\tau^{\nicefrac{(1-s)}{2p}}.
    \end{aligned}
\end{equation*}
Now, we estimate the terms in $G_T$.
Note that $\|Q_{\rm,0}-\tQ_{\rm,1}\|_{\rL^\infty_\tau(\rL^\infty_x)}\to 0, \tas \tau \to 0 $ in~\eqref{eq:est of first factor first term G_u} and then observe that 
\begin{equation*}
   \left\|  \Bigl(1-\frac{\tQ_{\rm,1}}{Q_{\rm,0}}\Bigr)\del_t\tT_1\right\|_{\rLp_\tau(\rLq_x)}\leq C\cdot \|Q_{\rm,0}-\tQ_{\rm,1}\|_{\rL^\infty_\tau(\rL^\infty_x)}\|\del_t\tT_1\|_{\rLp_\tau(\rLq_x)}.
\end{equation*} 
The desired estimate can be concluded along the same lines as \eqref{eq:est of first term of G_u}.
Next, consider $(1+\tT_1)(\tq_{\rv,1} - \tq_{\rvs})\tq_{\rc,1}$. For $\tT$, we employ \autoref{lem:ests of vars}(iii) and obtain
\begin{equation*}
    \begin{aligned}
       \left\|(1+\tT_1)(\tq_{\rv,1} - \tq_{\rvs})\tq_{\rc,1}\right\|_{\rL_\tau^p(\rL_x^{q})}
       &\leq \|1+\tT_1\|_{\rL_\tau^p(\rL_x^\infty)}\cdot\|\tq_{\rv,1} - \tq_{\rvs}\|_{\rL_\tau^\infty(\rL_x^q)}\cdot\|\tq_{\rc,1}\|_{\rL^\infty_\tau(\rL^\infty_x)}\\&\leq C (\tau^{\nicefrac{1}{p}}+\tau^{\nicefrac{(1-s)}{2p}})\leq C\tau^\delta.
    \end{aligned}
\end{equation*}
Next, we estimate the singular term $(1+\tT_1)(\tq_{\rv,1} - \tq_{\rvs})^+$.
Arguing as above, and invoking the estimate $\|(\tq_{\rv,1}-\tq_{\rvs})^+\|_{\rL^\infty_\tau(\rLq_x)}\leq \|\tq_{\rv,1}-\tq_{\rvs}\|_{\rL^\infty_\tau(\rLq_x)}$, we infer that 
\begin{equation*}
    \begin{aligned}
       \left\|(1+\tT_1)(\tq_{\rv,1} - \tq_{\rvs})^+\right\|_{\rL_\tau^p(\rL_x^{q})}
       \leq \|1+\tT_1\|_{\rL_\tau^p(\rL_x^\infty)}\cdot\|(\tq_{\rv,1} - \tq_{\rvs})^+\|_{\rL_\tau^\infty(\rL_x^q)}
       \leq C (\tau^{\nicefrac{1}{p}}+\tau^{\nicefrac{(1-s)}{2p}})\leq C\tau^\delta.
    \end{aligned}
\end{equation*}
We observe that the other terms appearing in \( G_T \) can be estimated in the same way.
Similarly, for $G_\rv,G_\rc,G_\rr$, we only give the estimates of singular source term, that is,
 \begin{equation*}\|(\tq_{\rv,1}-\tq_\rvs)^{+}\|_{\rL_\tau^p(\rL_x^{q})}\leq \|\tq_{\rv,1}\|_{\rL_\tau^p(\rL_x^{q})}+C\tau^{\nicefrac{1}{p}}\leq C \tau^\delta, 
 \end{equation*} 
where we use \autoref{lem:ests of vars}(iv) for $\tq_\rv$.

Finally, we consider the boundary terms such as
\begin{equation*}
    (1-Z_{3,3}) \del_{y_3} \tT_{1} - Z_{1,3} \del_{y_2} \tT_{1} - Z_{2,3} \del_{y_3} \tT_{1} \eqqcolon B_{T,1}.
\end{equation*}
Due to the same shape, it is enough to consider the part $(1-Z_{3,3}) \del_{y_3} \tT_{1} $.
Using the trace theorem, employing \autoref{lem:ests of vars}(iii) for $\tT$, making use of \autoref{lem:ests of trafo}(iii) to estimate $\mathrm{Id}_3 - Z$, and invoking that $\tT_{1} \in \E_{1,T,\tau}$, we deduce
\begin{equation*}
    \begin{aligned}
        \| (1-Z_{3,3}) \del_{y_3} \tT_{1} \|_{\rL_\tau^p(\rB_{q,q}^{1 - \nicefrac{1}{q}}(G))}
        &\leq \| (1-Z_{3,3}) \del_{y_3} \tT_{1} \|_{\rL_\tau^p(\rH_x^{1 ,q})}
        &\leq C\tau^{\nicefrac{1}{p'}}.
    \end{aligned}
\end{equation*}
It remains to estimate this term in the second part of the norm of the boundary space. 
Applying \autoref{lem:est of bilin term on boundary} in conjunction with \autoref{lem:ests of trafo}(i) for $\mathrm{Id}_3-Z$ and \autoref{lem:est of normal derivative} for~$\del_{y_3} \tT$, we get
\begin{equation*}
    \begin{aligned}
        \| (1-Z_{3,3}) \del_{y_3} \tT_{1} \|_{\rF_{p,q}^{\nicefrac{1}{2}-\nicefrac{1}{2q}}(\rL_x^q(G))}
        &\leq C\tau^{\delta} \cdot \|\mathrm{Id}_3-Z\|_{\rH_\tau^{1,p}(\rB_{q,q}^{1 - \nicefrac{1}{q}}(G))} \cdot \|\del_{y_3}\tT_{1}\|_{\rF_{p,q}^{\nicefrac{1}{2}-\nicefrac{1}{2q}}(\rL_x^q(G))}\\
        &\leq C\tau^{\delta}\cdot (\|\tT_{1}(0)\|_{\rB_{q,p}^{2(1 - \nicefrac{1}{P})}(\Omega)}+\|\tT_{1}\|_{\E_{1,T,\tau}})
        \leq C\tau^{\delta}.
    \end{aligned}
\end{equation*}

The other transformed terms on the boundary $B_{T,i}(\tz_1),B_{\rv,i}(\tz_1),B_{\rc,i}(\tz_1),B_{\rr,i}(\tz_1)), i\in \{1,2\}$ can be handled analogously.
\end{proof}

In the same way as the preceding lemma, we obtain the following lemma on estimates of differences of the nonlinear terms.

\begin{lem}\label{lem: contraction}
Consider $p \in (2,\infty)$, $q \in (3,\infty)$ as well as $\tz_1$, $\tz_2 \in \E_{1,\tau}$, with $\tau \le \Tilde{\tau}$, where $\tz_k$, $k=1,2$, represents a solution to the linearized problem \eqref{eq:complete lin syst in Lagrangian coords} with data $(f_\rd^k,f_u^k,f_T^k,f_j^k)$, $(b_i^k,g_{T,i}^k,q_{j,i}^k)$, $(\rho_{\rd,0},u_0,T_0,q_{j,0})$, $j \in \{\rv,\rc,\rr\}$, $i=1,2$, satisfying the assumptions as made precise in \autoref{prop:opt data result moisture compre NSE}.

Then there is a constant $C(\tau) > 0$, merely depending on $p$, $q$ and $\Tilde{\tau}$ and fulfilling $C(\tau) \to 0$ as $\tau \to 0$, so that for the nonlinear terms $G(\tz_k)$, $B_i(\tz_k)$, $i=1,2$, $k=1,2$, from \autoref{sec:Lagrangian coords}, it holds that
\begin{equation*}
    \| G(\tz_1) - G(\tz_2) \|_{\E_{0,\tau}} + \sum_{i=1}^2 \| B_i(\tz_1) - B_i(\tz_2) \|_{\F_\tau} \le C(\tau) \cdot \| \tz_1 - \tz_2 \|_{\E_{1,\tau}}.
\end{equation*}
\end{lem}

\subsection{Proof of \autoref{thm:localWP}}
\

In this subsection, we complete the proof of the local strong well-posedness.

\begin{proof}{(Proof of \autoref{thm:localWP}).}
For $\tz_1 = (\trho_{\rd,1},\tu_1,\tT_1,\tq_{\rv,1},\tq_{\rc,1},\tq_{\rr,1})  \in \E_{1,\tau}$, by $\tz \coloneqq \Psi(\tz_1) \in \E_{1,\tau}$, with~$\tz = (\trho_\rd, \tu,\tT,\tq_\rv, \tq_\rc,\tq_\rr)$, we denote the solution to the linearized system \eqref{eq:complete lin syst in Lagrangian coords} with right-hand sides~$G$ and boundary conditions $B_i$ given in terms of~$\tz_1$. 
The existence of a solution is guaranteed by \autoref{prop:opt data result moisture compre NSE} and \autoref{lem: estimate non lin}, and it follows that $\Psi$ is a self-map. 
In order to show that $\Psi$ is also a contraction, consider $\tz_1$, $\tz_2 \in \E_{1,\tau}$. 
Then \autoref{prop:opt data result moisture compre NSE} and \autoref{lem: contraction} yield the estimate
\begin{equation*}
    \| \Psi(\tz_1) - \Psi(\tz_2) \|_{\E_{1,\tau}} =  \| \tz_1 - \tz_2 \|_{\E_{1,\tau}} \le C \cdot \| G(\tz_1) - G(\tz_2) \|_{\E_{0,\tau}} \le C(\tau) \cdot \| \tz_1 - \tz_2 \|_{\E_{1,\tau}},
\end{equation*}
where the constant $C>0$ is independent of time, since $\tz_1-\tz_2$ has zero initial value. 
The contraction property then follows by choosing $\tau$ appropriately small such that $C(\tau) \in (0,1)$.
As the time $\tau > 0$ is chosen sufficiently small, it follows from \autoref{lem:ests of trafo} that $X(t,\cdot)$ is a $\rC^1$-diffeomorphism from $\Omega$ to $\Omega$, so we obtain the solution $(\rho_\rd,u,T,q_\rv,q_\rc,q_\rr) \in \E_{1,\tau}$ in the Eulerian formulation upon invoking the inverse transformation $Y$.
\end{proof}

\section{Global Well-Posedness for Small Data}\label{sec: global}

In order to prove the global-in-time existence for small data and in the absence of gravity, we assume that $q_\rvs \equiv 0$ for simplicity of notation. 
In fact, any other constant $q_\rvs^*$ could be considered by choosing the steady state solution $q_\rv =q_\rvs^*$, see also \autoref{rem_ assu}.

We transform the system \eqref{eq:coupled moisture compr NSE detailed} by using the Lagrange transform introduced in~\eqref{eq:CP Lagrange}. 
However, we define the new quantities taking into account the constant steady state solution.
More precisely, we define
\begin{equation*}
    \begin{aligned}
        \trho_\rd(t,y)
        &= \rho_\rd(t,X(t,y))- \bar{\rho}_\rd, \enspace \tu(t,y) = u(t,X(t,y)), \enspace \tT(t,y) =T(t,X(t,y)) - \bar{T} ,\\
        \tq_\rv(t,y) &= q_\rv(t,X(t,y)), \enspace \tq_\rc(t,y) = q_\rc(t,X(t,y)) -1 \tand \tq_\rr(t,y) = q_\rr(t,X(t,y)),
    \end{aligned}
\end{equation*}
where $\bar{\rho}_\rd>0$ and $\bar{T} \geq0$. 
The reason for this modification is our intention to linearize the system \eqref{eq:coupled moisture compr NSE detailed} around the constant steady state given by $(\bar{\rho}_\rd, 0,\bar{T},0,1,0)$. 
Note that $\tQ_m(t,y) =Q_m(t,X(t,y)) -1$. 
Therefore, the resulting transformed system is given by
\begin{equation}\label{eq:Lagrangian coords global}
	\left\{	
    \begin{aligned}
        \del_t\trho_\rd+ \bar{\rho}_\rd \mdiv\tu
        &=G_\rd, &&\tin (0,\tau) \times \Omega,\\
		\del_t\tu-\frac{1}{ 2\bar{\rho}_\rd}\rL\tu + \frac{\bar{T}}{2 \bar{\rho}_\rd} \nabla \trho_\rd + \frac{\bar{T}}{2}\nabla \tq_\rv + \frac{1}{2}\nabla \tT 
        &=G_u, &&\tin (0,\tau) \times \Omega,\\
        \del_t\tT-\frac{1}{2}\Delta\tT + \frac{1}{2}\bar{T} \mdiv \tu - \frac{ (\bar{T}+1)}{2}\tq_\rv
        &=G_T, &&\tin (0,\tau) \times \Omega,\\
        \del_t\tq_\rv-(\Delta -1)\tq_\rv
        &=G_\rv, &&\tin (0,\tau) \times \Omega,\\
        \del_t\tq_\rc-\Delta\tq_\rc - \tq_\rv+\tq_\rr
        &=G_\rc, &&\tin (0,\tau) \times \Omega,\\
        \del_t\tq_\rr-(\Delta+1) \tq_\rr 
        &=G_\rr, &&\tin (0,\tau) \times \Omega.
		\end{aligned}\right.
\end{equation} 
Recalling the difference of the transformed Laplacian and the original Laplacian~$\cL_1 - \Delta$ as well as the difference of the transformed term associated with $\nabla \mdiv$ from \eqref{eq:transformed Laplacian minus Laplacian} and \eqref{eq:transformed second comp Lame}, respectively, we find that the transformed terms $G_\rd$, $G_u$, $G_T$, $G_\rv$, $G_\rc$ and $G_\rr$ on the right-hand side of \eqref{eq:Lagrangian coords global} are given by
\begin{equation*}
    G_\rd = -\trho_\rd\mdiv\tu-(\trho_\rd + \bar{\rho}_\rd )\nabla\tu:[Z^\top-\mathrm{Id}_3],
\end{equation*}
\begin{equation*}
    \begin{aligned}
        (G_{u})_{i}
        &= - \frac{\trho_\rd}{\bar{\rho}_\rd}\del_t\tu_{i} - \frac{(\tq_\rv + \tq_\rc + \tq_\rr) \trho_\rd}{2 \bar{\rho}_\rd} \dt \tu_i - \frac{\tq_\rv + \tq_\rc + \tq_\rr}{2} \dt \tu_i + \frac{\mu}{ 2\bar{\rho}_\rd} (\cL_1 - \Delta) \tu_i\\
        &\quad + \frac{\mu+\lambda}{ 2\bar{\rho}_\rd} \bigl(\cL_2 - \nabla \mdiv) \tu\bigr)_i\\
        &\quad-\frac{1}{2 \bar{\rho}_\rd}\Bigl((\tT + \bar{T} )(Z^\top (\nabla \trho_\rd) \tq_\rv )_i+ \tT (Z^\top \nabla \trho_\rd)_i + \bar{T} \bigl ((Z^\top - \mathrm{Id}_3 ) \nabla \trho_\rd \bigr )_i\Bigr)\\
        &\quad -\frac{1}{2 \bar{\rho}_\rd}\Bigl( (\tT + \bar{T}) (\trho_\rd Z^\top \nabla \tq_\rv )_i+ \bar{\rho}_\rd \tT (Z^\top \nabla \tq_\rv)_i +\bar{T}\bar{\rho}_\rd \bigl ( (Z^\top - \mathrm{Id}_3) \nabla \tq_\rv \bigr )_i  \Bigr) \\
        &\quad -\frac{1}{2 \bar{\rho}_\rd}\Bigl( (\trho_\rd + \bar{\rho}_\rd ) (\tq_\rv Z^\top \nabla \tT )_i + \trho_\rd (Z^\top \nabla \tT)_i + \bar{\rho}_\rd \bigl ( (Z^\top - \mathrm{Id}_3 )\nabla \tT \bigr )_i  \Bigr)\\
        &\quad - \frac{\trho_\rd + \bar{\rho}_\rd}{2\bar{\rho}_\rd}\tq_\rr\tV_\rr\sum_{l=1}^3\frac{\del \tu_i}{\del y_l}Z_{l,3} ,
    \end{aligned}
\end{equation*}
\begin{equation*}
    \begin{aligned}
        G_T
        &= -\Bigl( \frac{\tq_\rv + \tq_\rc + \tq_\rr}{2}\Bigr)\del_t\tT + \frac{1}{2}\biggl((\cL_1 - \Delta) \tT + \tq_\rr\tV_\rr\sum_{l=1}^3 \frac{\del\tT}{\del y_{l}} Z_{l,3}\\
        &\qquad - \bigl (\tT + \tq_\rv ( \tT + \bar{T}) \bigr) (\nabla \tu:Z^\top)-\bar{T}\nabla \tu:(Z^\top-I)\\
        &\qquad -( \tT + \bar{T} +1 ) (\tT +\bar{T}) \frac{1+\tq_\rv}{2 + \tq_\rv + \tq_\rc + \tq_\rr} (-\tq_\rv)^+ \tq_\rr \\
        & \qquad + (\bar{T}+1) ( \tq_\rv \tq_\rc + (\tq_\rv)^+ ) + \tT(\tq_\rv \tq_\rc + (\tq_\rv)^+ ) \biggr ),
    \end{aligned}
\end{equation*}
\begin{equation*}
    \begin{aligned}
        G_\rv
        &= (\cL_1 - \Delta) \tq_\rv +(\tT +\bar{T}) \frac{1+\tq_\rv}{2 + \tq_\rv + \tq_\rc + \tq_\rr} ( -\tq_\rv)^+ \tq_\rr-\tq_\rv \tq_\rc-(\tq_\rv)^{+},\\
        G_\rc
        &= (\cL_1 - \Delta) \tq_\rc+\tq_\rv\tq_\rc +(\tq_\rv)^{+}-(\tq_\rc)^{+}-\tq_\rc\tq_\rr, \tand\\
        G_\rr
        &= (\cL_1 - \Delta) \tq_\rr +(\tq_\rc)^{+}+\tq_\rc\tq_\rr- (\tT +\bar{T}) \frac{1+\tq_\rv}{2 + \tq_\rv + \tq_\rc + \tq_\rr} ( -\tq_\rv)^+ \tq_\rr\\
        &\quad +\tV_\rr\sum_{l=1}^3 \frac{\del \tq_\rr}{\del y_{l}}Z_{l,3}
        +\tq_\rr\sum_{l=1}^3 \frac{\del \tV_\rr}{\del y_{l}} Z_{l,3}
        +\frac{1}{\trho_\rd + \bar{\rho}_\rd}\tq_\rr \tV_\rr\sum_{l=1}^3 \frac{\del \trho_\rd }{\del y_{l}} Z_{l,3}.
    \end{aligned}
\end{equation*}
The homogeneous Dirichlet boundary conditions for the fluid do not change by the transformation. 
However, the transformed boundary conditions for the temperature and the densities are given by
\begin{equation*}
    \begin{aligned}
        \del_{y_3} \tT
        &= (1-Z_{3,3}) \del_{y_3} \tT - Z_{1,3} \del_{y_1} \tT - Z_{2,3} \del_{y_2} \tT \eqqcolon B_{T,1}, &&\ton \Gau,\\
        \del_{y_3} \tT
        &= (1-Z_{3,3}) \del_{y_3} \tT - Z_{1,3} \del_{y_1} \tT - Z_{2,3} \del_{y_2} \tT \eqqcolon B_{T,2}, &&\ton \Gal,\\
        \del_{y_3} q_j
        &= (1-Z_{3,3}) \del_{y_3} q_j - Z_{1,3} \del_{y_1} q_j - Z_{2,3} \del_{y_2} q_j \eqqcolon B_{q_j ,1}, &&\ton \Gau,\\
        \del_{y_3} q_j
        &=  (1-Z_{3,3}) \del_{y_3} q_j - Z_{1,3} \del_{y_1} q_j - Z_{2,3} \del_{y_2} q_j \eqqcolon B_{q_j ,2}, &&\ton \Gal, 
    \end{aligned}
\end{equation*}
where $j \in \{\rv,\rc,\rr\}$. Finally the initial data is given by
\begin{equation}\label{eq: initial data global}
    \begin{aligned}
        \trho_\rd(0)&= \rho_\rd(0) -\bar{\rho}_\rd, \enspace \tu(0) = u(0), \enspace \tT(0) = T(0) -\bar{T}, \\ 
        \tq_\rv(0) &= q_\rv(0), \enspace \tq_\rc(0) = q_\rc(0)-1, \enspace \tq_\rr (0) = q_\rr(0).
    \end{aligned}
\end{equation}

\subsection{Linearized compressible Navier-Stokes-moisture system}\label{subsec:lin NS M}
\ 

In this subsection, we discuss some properties of the operator associated with the linear system \eqref{eq:Lagrangian coords global}. 
Set $\rX_0 \coloneqq \rH^{1,q}_\per(\Omega) \times \rLq(\Omega)^3 \times \rLq(\Omega)^4$, and for the principal variable $\tz = (\trho_\rd, \tu, \tT,\tq_\rv,\tq_\rc, \tq_\rr)$, define the differential operator $\cA$ by 
\begin{equation*}
    \begin{aligned}
        &\Bigl( -\bar{\rho}_\rd \mdiv \tu ,- \frac{\bar{T}\nabla  \trho_\rd }{2 \bar{\rho}_\rd} + \frac{\rL \tu}{2\bar{\rho}_\rd } - \frac{\nabla \tT}{2} - \frac{\bar{T}\nabla \tq_\rv}{2},- \frac{\bar{T}\mdiv \tu}{2} + \frac{\Delta \tT }{2}  +\frac{\bar{T}+1}{2}\tq_\rv,\\
        &\quad (\Delta -1)\tq_\rv,\tq_\rv + \Delta \tq_\rc + \tq_\rr, (\Delta +1) \tq_\rr \Bigr)^\top,
    \end{aligned}
\end{equation*}
subject to the boundary conditions 
\begin{equation}\label{eq: bs operator}
    \tu = 0 \ton \Gau \cup \Gal \tand \dz \tT= \dz \tq_j =0 \ton \Gau \cup \Gal, \tfor j \in \{ \rv,\rc,\rr \}.
\end{equation}
For $q>1$, consider the $\rLq$-realization of $\cA$ in $\rX_0$, subject to \eqref{eq: bs operator} defined by 
\begin{equation*}
\begin{aligned}
    A \tz &\coloneqq \cA \tz = \begin{pmatrix}
        0 & -\bar{\rho}_\rd \mdiv & 0 & 0 & 0 & 0 \\
       - \frac{\bar{T}}{2 \bar{\rho}_\rd} \nabla & \frac{1}{2 \bar{\rho}_\rd}\rL & -\frac{1}{2} \nabla & -\frac{\bar{T}}{2} \nabla & 0 & 0 \\ 0 & -\frac{\bar{T}}{2} \mdiv & \frac{1}{2}\Delta & \frac{ \bar{T}+1 }{2}& 0 & 0 \\ 0 & 0 & 0 & \Delta -1 & 0 & 0 \\ 0 & 0 & 0 & 1  & \Delta & 1 \\ 
       0 & 0 & 0 & 0 & 0& \Delta+1  
    \end{pmatrix} \begin{pmatrix}
        \trho_\rd \\ \tu \\ \tT \\ \tq_\rv \\ \tq_\rc \\ \tq_\rr 
    \end{pmatrix}, \twith
\end{aligned}
\end{equation*}
\begin{equation*}
    \begin{aligned}
        \rD(A) \coloneqq \bigl\{&\tz \in \rH^{1,q}_\per(\Omega) \times \rH^{2,q}_\per(\Omega)^3 \times \bigl ( \rH_\per^{2,q}(\Omega) \bigr )^4 :\\
        &\tu = 0 \tand \dz \tT = \dz \tq_j =0 \ton \Gau \cup \Gal \tfor j \in \{\rv,\rc,\rr\} \bigr \}.
    \end{aligned}
\end{equation*}
For the investigation of the linear system \eqref{eq:Lagrangian coords global}, we first show that the operator~$-A$, up to a shift, is $\mathcal{R}-$sectorial of angle strictly less than $\pi/2$, which in particular implies maximal $\rLp$-regularity on $\rX_0$ thanks to the UMD property of the latter space due to $q \in (1,\infty)$.

\begin{prop}
    \label{prop: max reg A}
    Let $q \in (1,\infty)$. Then there exists $\omega>0$ such that the operator~$-A+\omega$ is $\mathcal{R}-$sectorial on $\rX_0$ of angle $\Phi_{-A+\omega} < \nicefrac{\pi}{2}$.
\end{prop}

\begin{proof}
    We split the operator $A$ into $A = A_0 +Q$, where
    \begin{equation*}A_0=
    \begin{pmatrix}
        0 & -\bar{\rho}_\rd \mdiv & 0 & 0 & 0 & 0 \\
       0 & \frac{1}{2 \bar{\rho}_\rd}\rL & 0 & 0 & 0 & 0 \\ 0 & 0 & \frac{1}{2}\Delta & 0 & 0 & 0 \\ 0 & 0 & 0 & \Delta  & 0 & 0 \\ 0 & 0 & 0 & 0  & \Delta & 0 \\ 
       0 & 0 & 0 & 0 & 0& \Delta  
    \end{pmatrix} \ \text{and } \ Q =  \begin{pmatrix}
        0 & 0 & 0 & 0 & 0 & 0 \\
       - \frac{\bar{T}}{2 \bar{\rho}_\rd} \nabla & 0 & -\frac{1}{2} \nabla & -\frac{\bar{T}}{2} \nabla & 0 & 0 \\ 0 & -\frac{\bar{T}}{2} \mdiv & 0 &  \frac{ \bar{T}+1 }{2} & 0 & 0 \\ 0 & 0 & 0 & -1 & 0 & 0 \\ 0 & 0 & 0 & 1  & 0 & 1 \\ 
       0 & 0 & 0 & 0 & 0& 1
    \end{pmatrix} .
    \end{equation*}
    Since $\bar{\rho}_\rd \mdiv$ is a bounded operator from $\rH_\per^{2,q}(\Omega)^3$ to $\rH_\per^{1,q}(\Omega)$, and the Lam\'e operator as well as the Laplacian are $\cR$-sectorial up to a shift, we see that there exists $\omega>0$ such that $-A_0+\omega$ is $\cR$-sectorial on $\rX_0$ of angle strictly less than $\pi/2$. 
    Next, we note that for all $\tz \in \rD(A)$ it holds that
    \begin{equation*}
        \| Q \tz \|_{\rX_0} \leq C \bigl ( \| \trho_\rd \|_{\rH^{1,q}(\Omega)} +  \| \tT \|_{\rH^{1,q}(\Omega)} +  \| \tq_\rv \|_{\rH^{1,q}(\Omega)}+\| \tq_\rr \|_{\rH^{1,q}(\Omega)} + \| \tu \|_{\rH^{1,q}(\Omega)}  \bigr).
    \end{equation*}
    Complex interpolation and Young's inequality then imply that for any $\eps>0$, we get
    \begin{equation*}
        \begin{aligned}
            \| Q \tz \|_{\rX_0} &\leq \eps \bigl ( \| \tu \|_{\rH^{2,q}(\Omega)} +\| \tT \|_{\rH^{2,q}(\Omega)} + \| \tq_\rv \|_{\rH^{2,q}(\Omega)} + \|\tq_\rr\|_{\rH^{2,q}(\Omega)}   \bigr )\\ &\quad  + C(\eps) \bigl ( \| \trho_\rd \|_{\rH^{1,q}(\Omega)} +  \| \tT \|_{\rLq(\Omega)} +  \| \tq_\rv \|_{\rLq(\Omega)} + \| \tu \|_{\rLq(\Omega)} +\| \tq_\rr \|_{\rLq(\Omega)}  \bigr) \\
            &\leq \eps \cdot \| A_0 \tz \|_{\rX_0} + C(\eps) \cdot \| \tz \|_{\rX_0}.
        \end{aligned}
    \end{equation*}
    It follows that $A$ can be treated as a relatively $A_0$-bounded perturbation with relative bound zero.
    Hence, standard perturbation theory, see for example \cite[Prop.~4.3]{DHP:03}, yields the existence of $\omega>0$ such that $-A+\omega$ is $\cR$-sectorial on $\rX_0$ of angle $\Phi_A < \pi/2$.
\end{proof}

\subsection{Nonlinear estimates}
\

For some fixed $\tau > 0$, and with the maximal regularity space $\E_{1,\tau}$ as made precise in \eqref{eq:full sol space}, and for~$\eps > 0$, we introduce the ball
\begin{equation}\label{eq:sol ball} 
    \mathbb{B}_{\eps, \tau}\coloneqq\left\{\tz = (\trho_{\rd},\tu, \tT,\tq_\rv,\tq_\rc,\tq_\rr) \in \E_{1,\tau} : \| \tz \|_{\E_{1,\tau}}\leq \eps\right\}.
\end{equation}

As a preparation for the nonlinear estimates, we collect useful estimates of the terms related to the transform.
Again, we note that the proof is standard, and we refer for example to \cite[Sec.~3.9]{HMTT:19} for more details.

\begin{lem}\label{lem:ests of trafo global wp}
Let $p \in (2,\infty)$ and $q \in (3,\infty)$, consider $\tz \in \mathbb{B}_{\eps, \tau}$ for some fixed~$\tau > 0$, where $\mathbb{B}_{\eps, \tau}$ has been made precise in \eqref{eq:sol ball}, and recall $X$ from \eqref{eq:shape of X} as well as $Z = [\nabla X]^{-1}$.  
\begin{enumerate}[(i)]
    \item Then there exists a constant $C = C(p,q,\Omega,\tau) > 0$ such that
    \begin{equation*}
        \sup_{t \in (0,\tau
    )} \| \nabla X(t,\cdot) - \mathrm{Id}_3 \|_{\rL_x^\infty} + \sup_{t \in (0,\tau)} \| \nabla X(t,\cdot) - \mathrm{Id}_3 \|_{\rH_x^{1,q}} \le C \eps.
    \end{equation*}
    If $\eps\in (0,\frac{1}{2C})$, we especially have $\| \nabla X(t,\cdot) - \mathrm{Id}_3 \|_{\rL_\tau^\infty(\rL_x^\infty)} \le \frac{1}{2}$, so $\nabla X(t,\cdot)$ is invertible, and $Z = [\nabla X]^{-1}$ is thus well-defined.
    \item Similar estimates as in (i) are valid for $Z$ and $Z^\top$, i.e., for some constant $C = C(p,q,\Omega,\tau) > 0$, we have
    \begin{equation*}
        \begin{aligned}
            \sup_{t \in (0,\tau)} \| Z(t,\cdot) - \mathrm{Id}_3 \|_{\rH_x^{1,q}} + \sup_{t \in (0,\tau)} \| Z^\top(t,\cdot) - \mathrm{Id}_3 \|_{\rH_x^{1,q}} 
            &\le C \eps \tand\\
            \| Z \|_{\rL_\tau^\infty(\rH_x^{1,q})} + \| Z^\top \|_{\rL_\tau^\infty(\rH_x^{1,q})} 
            &\le C.
        \end{aligned}
    \end{equation*}
    \item For all $j$, $k$, $l = 1,2,3$, there is $C = C(p,q,\Omega,\tau) > 0$ with $\bigl\| \frac{\del Z_{l,j}}{\del y_k} \bigr\|_{\rL_\tau^\infty(\rL_x^q)} \le C \eps$.
    \item There exists a constant $C = C(p,q,\Omega,\tau) > 0$ such that the estimate $\| Z - \Id_3 \|_{\rC_\tau^{0,\nicefrac{1}{p'}}(\rH_x^{1,q})} \le C \eps$ is valid, where $p' \in (1,2)$ is the H\"older conjugate of $p$, i.e., $\nicefrac{1}{p}+\nicefrac{1}{p'}=1$.
\end{enumerate}
\end{lem}

In the following, for fixed $\tau > 0$, we will also fix $\eps_0 = \eps_0(\tau) > 0$ such that $\eps_0 < \nicefrac{1}{2C}$, and the map $Z$ is thus well-defined by \autoref{lem:ests of trafo global wp}.

In the lemma below, we make precise some further relations to establish the nonlinear estimates.
It can be obtained in a similar way as \autoref{lem:ests of vars}, upon invoking suitable embeddings of the maximal regularity space $\E_{1,\tau}$.

\begin{lem}\label{lem:further ests global wp}
Consider $p \in (2,\infty)$, $q \in (3,\infty)$ and $(\trho_{\rd},\tu,\tT,\tq_\rv,\tq_\rc,\tq_\rr) \in \mathbb{B}_{\eps, \tau}$ for some fixed $\tau > 0$ and $\eps \in (0,\eps_0)$, with $\eps_0$ as fixed above.
Then there exists a constant $C = C(p,q,\Omega,\tau) > 0$ such that
\begin{enumerate}[(i)]
    \item the density $\trho_\rd$ satisfies the estimate $\| \trho_\rd \|_{\rL_\tau^\infty(\rH_x^{1,q})} \le C \eps$,
    \item the fluid velocity $\tu$ fulfills $\| \tu \|_{\rL_\tau^\infty(\rH_x^{1,q})} \le C\eps$,
    \item for the temperature $\tT$, it holds that $\| \tT \|_{\rL_\tau^\infty(\rH_x^{1,q})} \le C\eps$, and
    \item the estimates for $\tT$ as stated in~(iii) remain valid for the mixing ratios $\tq_\rv$, $\tq_\rc$ and $\tq_\rr$.
\end{enumerate}
\end{lem}

Before verifying the following nonlinear estimates, for $j \in \{\rv,\rc,\rr\}$, we provide the following auxiliary problem supplemented with the initial data \eqref{eq: initial data global}. 

\begin{equation}\label{eq:Lagrangian coords global proof}
	\left\{	
    \begin{aligned}
        \del_t\trho_\rd+ \bar{\rho}_\rd \mdiv\tu
        &= f_\rd, &&\tin (0,\tau) \times \Omega,\\
		\del_t\tu -\frac{1}{ 2\bar{\rho}_\rd}\rL\tu + \frac{\bar{T}}{2 \bar{\rho}_\rd} \nabla \trho_\rd + \frac{\bar{T}}{2}\nabla \tq_\rv + \frac{1}{2}\nabla \tT 
        &= f_u, &&\tin (0,\tau) \times \Omega,\\
        \del_t\tT -\frac{1}{2}\Delta\tT + \frac{1}{2}\bar{T} \mdiv \tu - \frac{ \bar{T}+1 }{2}(\tq_\rv + (\tq_\rv)^+) 
        &=  f_T, &&\tin (0,\tau) \times \Omega,\\
        \del_t\tq_\rv -(\Delta -1)\tq_\rv + (\tq_\rv)^+
        &= f_\rv  , &&\tin (0,\tau) \times \Omega,\\
        \del_t\tq_\rc-\Delta\tq_\rc - \tq_\rv - (\tq_\rv)^+ + (\tq_\rc)^++\tq_\rr
        &=f_\rc, &&\tin (0,\tau) \times \Omega,\\
        \del_t\tq_\rr -(\Delta+1) \tq_\rr - (\tq_\rc)^+
        &= f_\rr, &&\tin (0,\tau) \times \Omega, \\
          \tu = 0, \enspace \dz \tT = g_{T,1}, \enspace \dz \tq_j
        &= g_{j,1}, &&\ton (0,\tau) \times \Gau,\\
        \tu = 0, \enspace \dz \tT = g_{T,2}, \enspace \dz \tq_j
        &= g_{j,2}, &&\ton (0,\tau) \times \Gal.
		\end{aligned}\right.
\end{equation}

\begin{prop}\label{prop: nonlin est global}
Let $p \in (2,\infty)$, $q \in (3,\infty)$, and for $\eps \in (0,\eps_0)$, where $\eps_0$ is as made precise after \autoref{lem:ests of trafo global wp}, and for fixed $\tau > 0$, let $\tz \in \mathbb{B}_{\eps,\tau}$ be a solution to \eqref{eq:Lagrangian coords global proof} with data $(f_{\rd},f_u,f_T,f_j)$, $(g_{T,i},g_{j,i})$, $(\rho_{\rd,0},u_{0}, T_{0}, q_{j,0})$, $j\in \{v,c,r\}, i=1,2$, satisfying the assumptions from \autoref{prop:opt data result moisture compre NSE}, with the natural adjustments due to the shifts in the initial data. 

Then there is a constant $C = C(p,q,\tau,\Omega) > 0$ so that for the adjusted terms
$G_\rd(\tz)$, $G_u(\tz)$, $\tilde{G}_T(\tz) = G_T(\tz) - \frac{\bar{T}+1}{2} (\tq_{\rv})^+$, $\tilde{G}_\rv(\tz) = G_\rv(\tz) - (\tq_{\rv})^+$, $\tilde{G}_\rc(\tz) = G_\rc(\tz)+(\tq_{\rv})^+ - (\tq_{\rc})^+$ and $\tilde{G}_\rr(\tz) = G_\rr(\tz)+ (\tq_{\rc})^+$, where we will use $\Tilde{G} = (G_\rd,G_u,\tilde{G}_T,\tilde{G}_\rv,\tilde{G}_\rc,\tilde{G}_\rr)$, and for the boundary terms $B_{T,i}(\tz)$, $B_{j,i}(\tz)$, $j \in \{\rv,\rc,\rr\}$, $i=1,2$, which will be denoted by $B_i = (B_{T,i},B_{\rv,i},B_{\rc,i},B_{\rr,i})$, we have
\begin{equation*}
    \begin{aligned}
        \| \Tilde{G}(\tz) \|_{\E_{1,\tau}} + \sum_{i=1}^2 \| B_i(\tz))\|_{\F_\tau} \le  C\eps^{2}.
    \end{aligned}
\end{equation*}
\end{prop}

\begin{proof}
Proceeding as in the proof of \autoref{lem: estimate non lin}, we give here an idea of the procedure for the estimates. 
Let us start with the term $G_{\rd}$.  
From the Banach algebra structure of $\rL^\infty(0,\tau;\rH^{1,q}(\Omega))$, \autoref{lem:further ests global wp}(i) for the estimate $\trho_{\rd}$ as well as $\tz \in \mathbb{B_{\eps,\tau}}$, it first follows that 
\begin{equation*}
    \|\trho_\rd \mdiv\tu\|_{\rLp_{\tau}(\rH^{1,q}_x)}\leq \|\trho_{\rd}\|_{\rL^\infty_\tau (\rH^{1,q}_x)}\|\mdiv\tu\|_{\rL^p_\tau(\rH^{1,q}_x)}\leq C\eps^{2}.
\end{equation*}
Similarly, additionally invoking the boundedness of the constant $\bar{\rho}_\rd$, and using \autoref{lem:ests of trafo global wp}(ii) for an estimate of $Z^\top- \mathrm{Id}_3$, we obtain
\begin{equation*}
    \begin{aligned}
        \|(\trho_{\rd}+\bar{\rho}_\rd)\nabla\tu:[Z^\top-\mathrm{Id}_3]\|_{\rL^p_\tau(\rH^{1,q}_x)}
        \leq \|\trho_{\rd}+\bar{\rho}_\rd\|_{\rL^\infty_\tau(\rH^{1,q}_x)} \cdot \|\nabla\tu\|_{\rL^p_\tau(\rH^{1,q}_x)} \cdot \|Z^\top-\mathrm{Id}_3\|_{\rL^\infty_\tau(\rH^{1,q}_x)}\leq C\eps^{2}.
    \end{aligned}
\end{equation*}
The next nonlinear term under consideration is $G_u$. 
For the estimate of the first term of $G_u$, we make use of \autoref{lem:further ests global wp}(i) in order to estimate $\trho_{\rd}$ by $\eps$ and exploit~$\tz \in \mathbb{B_{\eps,\tau}}$ for $\dt \tu$ to derive that
\begin{equation*}
   \left\|\frac{\trho_\rd}{\bar{\rho}_\rd}\del_t\tu\right\|_{\rL^p_\tau(\rL^{q}_x)}\leq \left\|\frac{\trho_{\rd}}{\bar{\rho}_\rd}\right\|_{\rL^\infty_\tau(\rH^{1,q}_x)} \cdot \|\partial_t\tu\|_{\rL^p_\tau(\rL^{q}_x)}\leq C\eps^2.
\end{equation*}
Next, noticing again that $\tz \in \mathbb{B_{\eps,\tau}}$, and employing \autoref{lem:ests of trafo global wp}(ii) for estimating the terms $Z- \mathrm{Id}_3$ and~$Z$, we have 
\begin{equation*}
    \begin{aligned}
        \left\| \frac{\mu}{ 2\bar{\rho}_\rd}\sum_{j,k,l=1}^3\frac{\partial^{2}\tu_{i}}{\partial{y_{k}}\partial{y_{l}}}(Z_{k,j}-\delta_{k,j})Z_{l,j}  \right\|_{\rL_\tau^{p}(\rL^{q}_x)}
        \leq C \|\tu\|_{\rL^p_\tau(\rH^{2,q}_x)}\cdot\|Z-\mathrm{Id}_3\|_{\rL^{\infty}_\tau(\rH_x^{1,q})} \cdot \|Z\|_{\rL^{\infty}_\tau (\rH_x^{1,q})}\leq C\eps^2.
    \end{aligned}
\end{equation*}
Let us observe that the term $\sum_{j,k,l=1}^3\frac{\partial\tu_{i}}{\partial{y_{l}}} \frac{\partial Z_{l,j}}{\partial{y_{k}}}Z_{k,j}
$ as well as the other terms of a similar shape in $G_u$ can be handled in the exact same way.

The next term in $G_u$ to be estimated is $(\tT + \bar{T} )(-Z^\top (\nabla \trho_\rd) \tq_\rv )$. 
Note that $\tT+\bar{T}$ can be bounded, and $\nabla \trho_{\rd}$ can be estimated by $\eps$ by virtue of \autoref{lem:further ests global wp}(iii) and~(i), respectively, while the factor $Z^\top$ can be bounded by \autoref{lem:ests of trafo global wp}(ii).
On the other hand, with regard to $\tq_\rv$, we recall that $\tz \in \mathbb{B_{\eps,\tau}}$, so a concatenation of the previous arguments and the embedding $\rH^{1,q}(\Omega) \hookrightarrow \rL^\infty(\Omega)$ lead to
\begin{equation*}
    \begin{aligned}
        \|(\tT + \bar{T} )(-Z^\top (\nabla \trho_\rd) \tq_\rv )\|_{\rL_\tau^{p}(\rL^{q}_x)}
        \leq \|\tT + \bar{T}\|_{\rL_\tau^{\infty}(\rH^{1,q}_x)} \cdot \|Z^\top\|_{\rL_\tau^{\infty}(\rH^{1,q}_x)} \cdot \|\nabla \trho_\rd\|_{\rL_\tau^{\infty}(\rL^{q}_x)} \cdot \|\tq_\rv\|_{\rL_\tau^{p}(\rH^{2,q}_x)}\le C\eps^2.
    \end{aligned}
\end{equation*}
The last term of $G_u$ to be taken into account is $\frac{\trho_\rd + \bar{\rho}_\rd}{2\bar{\rho}_\rd}\tq_\rr \sum_{l=1}^3\frac{\del \tu_i}{\del y_l}Z_{l,3}$.  
By \autoref{lem:ests of trafo global wp}(ii), \autoref{lem:further ests global wp}(i) as well as~(iv), we have 
\begin{equation*}
    \begin{aligned}
        \left\| - \frac{\trho_\rd + \bar{\rho}_\rd}{2\bar{\rho}_\rd}\tq_\rr \sum_{l=1}^3\frac{\del \tu_i}{\del y_l}Z_{l,3}\right\|_{\rL_\tau^{p}(\rL^{q}_x)}
        \leq ~\left\|\frac{\trho_\rd + \bar{\rho}_\rd}{2\bar{\rho}_\rd}\right\|_{\rL_\tau^{\infty}(\rH^{1,q}_x)} \cdot \|\tq_\rr\|_{\rL_\tau^{\infty}(\rH^{1,q}_x)} \cdot \|\tu\|_{\rL_\tau^{p}(\rH^{2,q}_x)} \cdot \|Z\|_{\rL_\tau^{\infty}(\rH^{1,q}_x)}\leq C\eps^2.
    \end{aligned}
\end{equation*}
In the same way, this time employing \autoref{lem:further ests global wp}(iv) and $\tz  \in \mathbb{B_{\eps,\tau}}$ for $\tT$, we obtain
 \begin{equation*}
    \left\| \frac{\tq_\rv + \tq_\rc + \tq_\rr}{2}\del_t\tT\right\|_{\rL_\tau^{p}(\rL^{q}_x)}\leq\left\|\frac{\tq_\rv + \tq_\rc + \tq_\rr}{2}\right\|_{\rL_\tau^{\infty}(\rH^{1,q}_x)} \cdot \|\del_t\tT\|_{\rL_\tau^{p}(\rL^{q}_x)}\leq C\eps^2.
 \end{equation*}
Making use of the embeddings $\E_{1,T,\tau}\hookrightarrow \rL^\infty(\Omega \times (0,\tau))$ and $\E_{1,\rv,\tau} \hookrightarrow \rL^\infty(\Omega \times (0,\tau))$, see also the proof of \autoref{lem:further ests global wp}, and invoking \autoref{lem:further ests global wp}(iv) as well as $\tz  \in \mathbb{B_{\eps,\tau}}$ for $\tq_\rr$, we derive the estimate
 \begin{equation*}\begin{aligned}
    & \quad \left\| ( \tT + \bar{T} +1 ) (\tT +\bar{T}) \frac{1+\tq_\rv}{2 + \tq_\rv + \tq_\rc + \tq_\rr} ( \tq_\rv)^+ \tq_\rr \right\|_{\rL_\tau^{p}(\rL^{q}_x)}\\
    &\leq\|\tT + \bar{T} +1 \|_{\rL_\tau^{\infty}(\rL^{\infty}_x)} \cdot \|\tT +\bar{T}\|_{\rL_\tau^{\infty}(\rL^{\infty}_x)}\cdot \left\|\frac{1+\tq_\rv}{2 + \tq_\rv + \tq_\rc + \tq_\rr}\right\|_{\rL_\tau^{\infty}(\rL^{\infty}_x)}\cdot \|\tq_\rv\|_{\rL_\tau^{\infty}(\rH^{1,q}_x)} \cdot \|\tq_\rr\|_{\rL_\tau^{p}(\rL^{q}_x)} \le C\eps^2.
 \end{aligned}
 \end{equation*}
Next, we estimate the term $\tV_\rr\sum_{l=1}^3 \frac{\del \tq_\rr}{\del y_{l}}Z_{l,3}
        $. 
Employing that $\tz \in \mathbb{B_{\eps,\tau}}$ for $\tq_\rr$, and invoking the assumption on $\tV_\rr$, we find that 
 \begin{equation*}
     \left\|\tV_\rr\sum_{l=1}^3 \frac{\del \tq_\rr}{\del y_{l}}Z_{l,3}\right\|_{\rL_\tau^{p}(\rL^{q}_x)}\leq\| \tV_\rr\|_{\rL_\tau^{\infty}(\rL^{\infty}_x)}\cdot \|Z_{l,3}\|_{\rL_\tau^{\infty}(\rL^{\infty}_x)}\cdot\|\tq_\rr\|_{\rL_\tau^{p}(\rH^{1,q}_x)}\leq C\eps^{2}.
\end{equation*}
We can deal with $\tq_\rr\sum_{l=1}^3 \frac{\del \tV_\rr}{\del y_{l}} Z_{l,3}$ in the same way.\\
Let us observe that the other terms appearing in $G_u,\tilde{G}_T, \tilde{G}_\rv, \tilde{G}_\rr, \tilde{G}_\rc$ can be estimated in the same way as above, and it thus remains to estimate the boundary terms.

For the boundary terms, we first use standard trace theory together with the Banach algebra structure of $\rL^\infty(0,\tau;\rW^{1,q}(\Omega))$ by $q > 3$, \autoref{lem:ests of trafo global wp}(ii) to estimate $\Id_3 - Z$ and $\tz \in \mathbb{B}_{\eps,\tau}$ for $\tT$ to derive that
\begin{equation*}
    \begin{aligned}
        \| (\Id_3 - Z) \nabla \tT \|_{\rL_\tau^p(\rW^{1-\nicefrac{1}{q},q}(G))} 
        &\le C \cdot \| (\Id_3 - Z^\top) \nabla \tT \|_{\rL_\tau^p(\rH_x^{1,q})}\\
        &\le C \cdot \| (\Id_3 - Z^\top) \|_{\rL_\tau^\infty(\rH_x^{1,q})} \cdot \| \nabla \tT \|_{\rL_\tau^p(\rH_x^{1,q})} \le C \eps^2.
    \end{aligned}
\end{equation*}
This also yields that $\| B_{T,1} \|_{\rL_\tau^p(\rW^{1-\nicefrac{1}{q},q}(G))} \le C \eps^2$, and likewise for $B_{T,2}$.
On the other hand, by~$\tT \in \E_{1,T,\tau}$, we deduce similarly as in \cite[Prop.~6.4]{DHP:07} that 
\begin{equation*}
    \left\| \left.\del_{k} \tT \right|_{G} \right\|_{\rF_{p,q,\tau}^{\nicefrac{1}{2}-\nicefrac{1}{2q}}(\rL^q(G))} \le C \cdot \| \tT \|_{\E_{1,T,\tau}} \le C \eps.
\end{equation*}
From standard trace theory together with \autoref{lem:ests of trafo global wp}(iv), it further follows that
\begin{equation*}
    \left\| \left.(\Id_3 - Z) \right|_{G} \right\|_{\rC_\tau^{0,\nicefrac{1}{p'}}(\rW^{1-\nicefrac{1}{q},q}(G))} \le C \cdot \left\| Z^\top - \Id_3 \right\|_{\rC_\tau^{0,\nicefrac{1}{p'}}(\rH_x^{1,q})} \le C \eps.
\end{equation*}
Making use of a paraproduct estimate, see \cite[Thm.~2.8.2(ii)]{Tri:10}, we find that
\begin{equation*}
    \begin{aligned}
        \| B_{T,1} \|_{\rF_{p,q,\tau}^{\nicefrac{1}{2}-\nicefrac{1}{2q}}(\rL^q(G))}
        \le C \cdot \left\| \left.\del_{k} \tT \right|_{G} \right\|_{\rF_{p,q,\tau}^{\nicefrac{1}{2}-\nicefrac{1}{2q}}(\rL^q(G))} \cdot \left\| \left.(\Id_3 - Z) \right|_{G} \right\|_{\rC_\tau^{0,\nicefrac{1}{p'}}(\rW^{1-\nicefrac{1}{q},q}(G))} \le C \eps^2.
    \end{aligned}
\end{equation*}
Note that $B_{T,2}$ can be estimated in the same way.
The proof of this proposition follows upon noting that the other boundary terms can also be estimated by the same procedure.
\end{proof}

Finally, in order to deal with the nonlinear terms accounting for the phase transitions, we consider the following auxiliary problem for some given time $\tau>0$:
\begin{equation}\label{eq: auxillary for q}
    \left\{
    \begin{aligned}
        \dt Q - \Delta Q
        &= f -Q^+, &&\tin (0,\tau) \times \Omega,\\
        \dz Q
        &= g_u, &&\ton (0,\tau) \times \Gau,\\
       \dz Q
        &= g_l, &&\ton (0,\tau) \times \Gal,\\
        Q(0)
        &= Q_0, &&\tin \Omega,
    \end{aligned}
    \right.
\end{equation}
and assume that $Q \in \E_{1,h,\tau}$ is a local, maximal solution for given data $(f,g_u,g_l,Q_0)$. 
In the following, we derive a priori bounds for the solution $Q$ depending on the size of the data as made precise below.

\begin{prop}
    \label{prop: est on q}
    Let $\tau >0$ and $p,q \in [2,\infty)$. Assume that for $i \in \{ u,l\}$, the data satisfies
    \begin{equation}\label{eq: assu q}
        \begin{aligned}
            &f \in \E_{0,h,\tau}, \ g_i \in \F_{h,\tau}   \tand Q_0 \in \rB_{q,p,\per}^{2(1-\nicefrac{1}{p})}(\Omega) \twith\\
            &\| f \|_{\E_{0,h,\tau}} + \| g_i \|_{\F_{h,\tau}} +\|Q_0 \|_{\rB^{2(1-\nicefrac{1}{p})}_{q,p}(\Omega) } \leq C_0\eps^2 ,
        \end{aligned}
    \end{equation}
    for some $\eps>0$ and $C_0>0$. 
    Then there is a constant $C>0$ with $\| Q \|_{\E_{1,h,\tau}} \leq C \eps^2$.
\end{prop}

\begin{proof}
For brevity, for $q \in [1,\infty]$, we also simply denote the spatial $q$-norm by~$\| \cdot \|_q$.
We multiply \eqref{eq: auxillary for q} by $Q|Q|^{q-2}$ and integrate in space to obtain
\begin{equation*}
    \begin{aligned}
        \dt \| Q \|^q_q + \| \nabla | Q |^{\frac{q}{2}} \|^2_2 + \| |\nabla Q|  | Q|^{\frac{q-2}{2}} \|^2_2
        \leq C   \Bigl ( \int_\Omega \bigl (f - Q^+ \bigr )\cdot Q |Q|^{q-2} \srd x  + \int_G (g_u -g_l) \cdot Q |Q|^{q-2} \srd S\Bigr )
    \end{aligned}
\end{equation*}
for some constant $C>0$. 
In view of H\"older's and Young's inequality, we get
\begin{equation*}
    \Bigl | \int_\Omega (f - Q^+)\cdot Q |Q|^{q-2} \srd x \Bigr | \leq C \bigl ( \| f \|_q \cdot  \| Q \|_q^{q-1} + \| Q \|^q_q \bigr ) \leq C \bigl ( \| f\|^q_q + \| Q\|^q_q \bigr).
\end{equation*}
Concerning the boundary terms, we observe that similar arguments yield
\begin{equation*}
    \begin{aligned}
        \Bigl | \int_G (g_u -g_l) \cdot Q |Q|^{q-2} \srd S\Bigr | 
        \leq C \cdot \| g_u -g_l \|_{\rLq(G)}^q + \eps \cdot \| Q \|^q_{\rLq(G)}
        \leq C \cdot \| g_u -g_l \|^q_{\rB_{q,q}^{1 - \nicefrac{1}{q}}(G)} + \eps \cdot \| |Q|^{\frac{q}{2}} \|^2_{\rL^2(G)}.
    \end{aligned}
\end{equation*}
The trace theorem, interpolation and Young's inequality then imply
\begin{equation*}
    \begin{aligned}
        \| |Q|^{\frac{q}{2}} \|^2_{\rL^2(G)} \leq C \| |Q|^{\frac{q}{2}} \|^2_{\rW^{\nicefrac{1}{2},2}(\Omega)} 
        \leq C \bigl (\| |Q|^{\frac{q}{2}} \|^2_{\rH^{1}(\Omega)} + \||Q|^{\frac{q}{2}} \|^2_2 \bigr )
        \leq C\bigl ( \|\nabla |Q|^{\frac{q}{2}} \|^2_2 + \| Q\|^q_q \bigr ). 
    \end{aligned}
\end{equation*}
Hence, after absorbing, we obtain 
\begin{equation*}
    \dt \| Q \|^q_q \leq C \Bigl ( \| f\|^q_q + \| g_u -g_l \|^q_{\rB_{q,q}^{1 - \nicefrac{1}{q}}(G)}+ \| Q\|^q_q \Bigr). 
\end{equation*}
An integration in time then leads to
\begin{equation*}
    \| Q \|^q_q \leq C \bigl (\| f \|^q_{\rLq_\tau(\rL_x^q)} +\| g_u -g_l \|^q_{\rLq_\tau(\rB_{q,q}^{1 - \nicefrac{1}{q}}(G))}\bigr )+ C \int_0^t \| Q \|^q_q \ \mathrm{d}s + \| Q_0 \|^q_q.
\end{equation*}
By Gronwall's inequality, we conclude that
\begin{equation*}
    \| Q \|^q_q \leq \Bigl ( C  \| f \|^q_{\rLq_\tau(\rL_x^q)} +C\| g_u -g_l \|^q_{\rLq_\tau(\rB_{q,q}^{1 - \nicefrac{1}{q}}(G))} + \| Q_0 \|^q_q \Bigr ) \mathrm{e}^{C\tau}.
\end{equation*}
By Sobolev embeddings, we get $\rB^{2(1-\nicefrac{1}{p})}_{q,p,\per}(\Omega) \hookrightarrow \rLq(\Omega)$, so $\| Q\|_{\rL^\infty_\tau(\rL_x^q)} \leq C \eps^2$ by the assumptions on the data.
In particular, for all $q \geq 2$, we have 
\begin{equation*}
     \| \nabla | Q |^{\frac{q}{2}} \|^{\nicefrac{q}{2}}_2 + \| |\nabla Q|  | Q|^{\frac{q-2}{2}} \|^{\nicefrac{q}{2}}_2 \leq C \eps^2.
\end{equation*}
By the maximal $\rLp$-regularity on $\rX_0$, there exists a constant $C>0$ such that the local solution $Q \in \E_{1,h,\tau}$ to~\eqref{eq: auxillary for q} can be estimated by 
   \begin{equation*}
       \| Q \|_{\E_{1,h,\tau}} \leq C \bigl ( \| f \|_{\E_{0,h,\tau}} + \| Q^+\|_{\E_{0,h,\tau}}
       + \| g_u \|_{\F_{h,\tau}} +\| g_l \|_{\F_{h,\tau}} + \| Q_0 \|_{\rB_{q,p}^{2(1-\nicefrac{1}{p})}(\Omega)} \bigr ).
   \end{equation*}
The previous step joint with the embedding $\rL^\infty(0,\tau) \hookrightarrow \rLp(0,\tau)$ then yields
\begin{equation*}
    \| Q^+\|_{\E_{0,h,\tau}} \leq C \cdot \| Q \|_{\rL^\infty(0,\tau;\rLq(\Omega)) } \leq C\eps^2.
\end{equation*}
The claimed estimate then follows from the assumptions on the data.
\end{proof}

We are now in the position to prove the global existence of strong solutions for small data. 

\begin{proof}{(Proof of \autoref{thm: global WP}).}
Let $\tau>0$, $(f_\rd, f_u, f_T,f_\rv,f_\rc,f_\rr) \in \E_{0,\tau}$ and $(g_{T,i},g_{\rv,i},g_{\rc,i},g_{\rr,i}) \in (\F_{h,\tau})^4$ for $i \in \{1,2\}$ . 
Consider the system \eqref{eq:Lagrangian coords global proof}  supplemented with the initial data \eqref{eq: initial data global}. 
Then by \autoref{prop: max reg A} and similar arguments as in \autoref{sec: local}, there is a unique, local, maximal strong solution $\tz = (\trho_\rd,\tu,\tT,\tq_\rv,\tq_\rc,\tq_\rr) \in \E_{1,a_\mathrm{max}}$ of \eqref{eq:Lagrangian coords global proof}, where $a_{\mathrm{max}}>0$ denotes the maximal time of existence. 
If $a_{\mathrm{max}} < \tau$, i.e., the solution does not exist globally, this time $a_{\mathrm{max}}$ is characterized by the condition
\begin{equation*}
    \lim\limits_{t \to a_{\mathrm{max}}} \| (\trho_\rd,\tu,\tT,\tq_\rv,\tq_\rc,\tq_\rr) \|_{\E_{1,t}} = \infty.
\end{equation*}
Now for $\tz_1 \in \mathbb{B}_{\eps,a_{\mathrm{max}}} $, consider \eqref{eq:Lagrangian coords global proof} with $f_{\rd} = G_\rd(\tz_1)$, $f_{u} = G_u(\tz_1)$, $f_T = \Tilde{G}_T(\tz_1)$, $f_\rv = \Tilde{G}_\rv(\tz_1)$, $f_\rc = \Tilde{G}_{\rc}(\tz_1)$, $f_\rr = \Tilde{G}_\rr(\tz_1)$ and boundary terms 
\begin{equation*}
    g_{T,i} = B_{T,i}(\tz_1), \enspace g_{\rv,i} = B_{\rv,i}(\tz_1), \enspace g_{\rc,i} = B_{\rc,i}(\tz_1) \tand g_{\rr,i} = B_{\rr,i}(\tz_1), \tfor i \in \{1,2\}.
\end{equation*}
For the solution map 
\begin{equation*}
    \Phi : \mathbb{B}_{\eps,a_{\mathrm{max}}} \to \mathbb{B}_{\eps,a_{\mathrm{max}}}, \enspace \tz_1=(\trho_{\rd,1},\tu_1,\tT_1,\tq_{\rv,1},\tq_{\rc,1},\tq_{\rr,1})  \mapsto \tz =(\trho_\rd,\tu,\tT,\tq_\rv,\tq_\rc,\tq_\rr),
\end{equation*} 
we show in the following that $\Phi$ is well-defined, a self-map and a contraction, given that $\eps>0$ is sufficiently small. 
By \autoref{prop: nonlin est global}, it holds that  
\begin{equation*}
    \| \Tilde{G}(\tz_1) \|_{\E_{0,a_{\mathrm{max}}}} + \sum_{i=1}^2 \| (B_{T,i}(\tz_1), B_{\rv,i}(\tz_1),B_{\rc,i}(\tz_1), B_{\rr,i}(\tz_1)) \|_{\F_{a_\mathrm{max}}} \leq C\eps^2.
\end{equation*}
Combining the maximal $\rLp$-regularity on $\rX_0$ with \autoref{prop: est on q}, for all $t \in [0,a_{\mathrm{max}}]$, we then obtain
\begin{equation*}
    \begin{aligned}
        \| (\trho_\rd,\tu,\tT,\tq_\rv,\tq_\rc,\tq_\rr) \|_{\E_{1,t}}
        \leq C \eps^2 + \| ( \trho_{\rd,0}- \bar{\rho}_\rd, \tu_0, \tT_0 - \bar{T}, \tq_{\rv,0}, \tq_{\rc,0}-1, \tq_\rr ) \|_{\rX_\gamma} \leq C\eps^2,
    \end{aligned}
\end{equation*}
where we also used the assumptions on the initial data in the last inequality. 
Now, choosing $\eps>0$ sufficiently small yields that 
\begin{equation*}
    \| (\trho_\rd,\tu,\tT,\tq_\rv,\tq_\rc,\tq_\rr) \|_{\E_{1,t}} \leq \eps \enspace \text{for all} \enspace t \in[0,a_{\mathrm{max}}],
\end{equation*}
which shows that $\Phi$ is well-defined and a self-map. The contraction property can be shown in the same way. 
Finally, since $\eps$ may be chosen sufficiently small, the transform $X(t,\cdot)$ is a $\rC^1$-diffeomorphism, see \autoref{lem:ests of trafo global wp}, from $\Omega$ into $\Omega$ and we define 
\begin{equation*}
    \begin{aligned}
        \rho_\rd(t,x)
        &= \trho_\rd(t,Y(t,x))+ \bar{\rho}_\rd, \enspace u(t,x) = \tu(t,Y(t,x)), \enspace T(t,x) =T(t,Y(t,x)) + \bar{T} ,\\
        q_\rv(t,x) &= \tq_\rv(t,Y(t,x)), \enspace q_\rc(t,x) = \tq_\rc(t,Y(t,x)) +1 , \enspace q_\rr(t,x) = q_\rr(t,Y(t,x)).
    \end{aligned}
\end{equation*}
Then $(\rho_\rd,u,T,q_\rv, q_\rc,q_\rr)$ is a global, strong solution to system \eqref{eq:coupled moisture compr NSE detailed} with $(\rho_\rd,u,T,q_\rv, q_\rc,q_\rr)\in \E_{1,\tau}$.
\end{proof}

\medskip 

{\bf Acknowledgements.}
Felix Brandt would like to thank the German National Academy of Sciences Leopoldina for support via the Leopoldina Fellowship Program with grant number LPDS~2024-07.
Matthias Hieber and Tarek Z\"ochling acknowledge the support by DFG project FOR~5528.
Moreover, the authors would like to thank Rupert Klein for fruitful discussions.
They are also grateful to the anonymous referee for the careful reading of the manuscript and valuable comments.


\begin{thebibliography}{00}

\bibitem{Ama:95}
H.~Amann,
{\it Linear and Quasilinear Parabolic Problems}. Monographs in Mathematics, vol.~89, Birkh\"auser, 1995.
			
\bibitem{Ama:19}
H.~Amann,
{\it Linear and Quasilinear Parabolic Problems. Vol.~II. Function Spaces}. Monographs in Mathematics, vol.~106, Birkh\"auser/Springer, Cham, 2019.

\bibitem{BCD:11}
H.~Bahouri, J.-Y.~Chemin, R.~Danchin, 
{\it Fourier Analysis and Nonlinear Partial Differential Equations}.
{Grundlehren der math.~Wissenschaften, 343}, Springer, Heidelberg, 2011.

\bibitem{BCZT:14}
A.~Bousquet, M.~Coti-Zelati, R.~Temam, 
Phase transition models in atmospheric dynamics.
{\it Milan J. Math.}~{\bf 82} (2014), 99--128.

\bibitem{BD:07}
D.~Bresch, B.~Desjardins, 
On the existence of global weak solutions to the Navier-Stokes equations for viscous compressible and heat conducting fluids.
{\it J. Math. Pures Appl.~(9)}~{\bf 87} (2007), 57--90.

\bibitem{CT:07}
C.~Cao, E.S.~Titi, 
Global well-posedness of the three-dimensional viscous primitive equations of large scale ocean and atmosphere dynamics.
{\it Ann. of Math.~(2)}~{\bf 166} (2007), 245--267.

\bibitem{CZFTT:13}
M.~Coti-Zelati, M.~Fr\'emond, R.~Temam, J.~Tribbia,
The equations of the atmosphere with humidity and saturation: uniqueness and physical bounds.
{\it Phys. D}~{\bf 264} (2013), 49--65.

\bibitem{CZHKTZ:15}
M.~Coti-Zelati, A.~Huang, I.~Kukavica, R.~Temam, M.~Ziane,
The primitive equations of the atmosphere in presence of vapour saturation.
{\it Nonlinearity}~{\bf 28} (2015), 625--668.

\bibitem{CZT:12}
M.~Coti-Zelati, R.~Temam,
The atmospheric equation of water vapor with saturation.
{\it Boll. Unione Mat. Ital.~(9)}~{\bf 5} (2012), 309--336.

\bibitem{Dan:14}
R.~Danchin,
A Lagrangian approach for the compressible Navier-Stokes equations.
{\it Ann. Inst. Fourier (Grenoble)} {\bf 64} (2014), 753--791.

\bibitem{DHP:03} 
R.~Denk, M.~Hieber, J.~Pr\"uss,
{\it $\mathcal{R}$-Boundedness, Fourier Multipliers and Problems of Elliptic and Parabolic Type}. 
{Mem. Amer. Math. Soc.}~{\bf 166} 2003, no.~788.

\bibitem{DHP:07}
R.~Denk, M.~Hieber, J.~Pr\"uss, 
Optimal $\rLp$-$\rLq$-estimates for parabolic boundary value problems with inhomogeneous data.
{\it Math. Z.}~{\bf 257} (2007), 193--224.

\bibitem{DKLT:24}
S.~Doppler, R.~Klein, X.~Liu, E.S.~Titi,
Local well-posedness of a system coupling compressible atmospheric dynamics and a microphysics model of moisture in air.
{\it SIAM J. Math. Anal.}~{\bf 57} (2025), 5161--5188.

\bibitem{Fei:04}
E.~Feireisl,
{\it Dynamics of Viscous Compressible Fluids}. Oxford Lecture Series in Math.~and its Appl., Vol.~26, Oxford University Press, Oxford, 2004.

\bibitem{FGGVN:12}
E.~Feireisl, I.~Gallagher, D.~Gerard-Varet, A.~Novotn\'y,
Multi-scale analysis of compressible viscous and rotating fluids.
{\it Comm. Math. Phys.}~{\bf 314} (2012), 641--670.

\bibitem{FMP:04}
D.M.W.~Frierson, A.J.~Majda, O.M.~Pauluis,
Large scale dynamics of precipitation fronts in the tropical atmosphere: a novel relaxation limit.
{\it Commun. Math. Sci.}~{\bf 2} (2004), 591--626.

\bibitem{GS:96}
W.W.~Grabowski, P.K.~Smolarkiewicz, 
Two-time-level semi-Lagrangian modeling of precipitating clouds.
{\it Mon. Weather Rev.}~{\bf 124} (1996), 487--497.

\bibitem{HMTT:19}
B.~Haak, D.~Maity, T.~Takahashi, M.~Tucsnak, 
Mathematical analysis of the motion of a rigid body in a compressible Navier-Stokes-Fourier fluid.
{\it Math. Nachr.}~{\bf 292} (2019), 1972--2017.

\bibitem{HK:16}
M.~Hieber, T.~Kashiwabara,
Global strong well-posedness of the three dimensional primitive equations in $\rL^p$-spaces.
{\it Arch. Rational Mech. Anal.} {\bf 221} (2016), 1077--1115. 

\bibitem{HK:18}
S.~Hittmeir, R.~Klein, 
Asymptotics for moist deep convection~I: refined scalings and self-sustaining updrafts.
{\it Theor. Comput. Fluid Dyn.}~{\bf 32} (2018), 137--164.

\bibitem{HKLT:17}
S.~Hittmeir, R.~Klein, J.~Li, E.S.~Titi,
Global well-posedness for passively transported nonlinear moisture dynamics with phase changes.
{\it Nonlinearity}~{\bf 30} (2017), 3676--3718.

\bibitem{HKLT:20}
S.~Hittmeir, R.~Klein, J.~Li, E.S.~Titi,
Global well-posedness for the primitive equations coupled to nonlinear moisture dynamics with phase changes.
{\it Nonlinearity}~{\bf 33} (2020), 3206--3236.

\bibitem{HKLT:23}
S.~Hittmeir, R.~Klein, J.~Li, E.S.~Titi, 
Global well-posedness for the thermodynamically refined passively transported nonlinear moisture dynamics with phase changes.
{\it J. Nonlinear Sci.}~{\bf 33} (2023), Paper No.~65.

\bibitem{Kes:69}
E.~Kessler, 
On the distribution and continuity of water substance in atmospheric circulations.
{\it Meteorol. Monogr.}~{\bf 10} (1969), 1--84.

\bibitem{Kho:19}
B.~Khouider,
{\it Models for Tropical Climate Dynamics - Waves, Clouds, and Precipitation}. Mathematics of Planet Earth, Vol.~3, Springer, Cham, 2019.

\bibitem{KM:06}
R.~Klein, A.~Majda, 
Systematic multiscale models for deep convection on mesoscales.
{\it Theor. Comput. Fluid Dyn.}~{\bf 20} (2006), 525--551.

\bibitem{LT:16}
J.~Li, E.S.~Titi,
A tropical atmosphere model with moisture: global well-posedness and relaxation limit.
{\it Nonlinearity}~{\bf 29} (2016), 2674--2714.

\bibitem{Lio:98}
P.L.~Lions,
{\it Mathematical Topics in Fluid Mechanics. Vol.~2. Compressible Models}. Oxford Lecture Series in Math.~and its Appl., Vol.~10, The Clarendon Press, Oxford University Press, New York, 1998.

\bibitem{MS:10}
A.J.~Majda, P.E.~Souganidis, 
Existence and uniqueness of weak solutions for precipitation fronts: a novel hyperbolic free boundary problem in several space variables.
{\it Comm. Pure Appl. Math.}~{\bf 63} (2010), 1351--1361.

\bibitem{MN:80}
A.~Matsumura, T.~Nishida, 
The initial value problem for the equations of motion of viscous and heat-conductive gases.
{\it J. Math. Kyoto Univ.}~{\bf 20} (1980), 67--104.

\bibitem{MN:83}
A.~Matsumura, T.~Nishida, 
Initial-boundary value problems for the equations of motion of compressible viscous and heat-conductive fluids.
{\it Comm. Math. Phys.}~{\bf 89} (1983), 445--464.

\bibitem{PCK:08}
O.~Pauluis, A.~Czaja, R.~Korty,
The global atmspheric circulation on moist isentropes.
{\it Science}~{\bf 321} (2008), 1051--1078.

\bibitem{RTSS:24a}
A.~Remond-Tiedrez, L.M.~Smith, S.N.~Stechmann, 
A nonlinear elliptic PDE from atmospheric science: well-posedness and regularity at cloud edge.
{\it J. Math. Fluid Mech.}~{\bf 26} (2024), Paper No.~30.

\bibitem{RTSS:24b}
A.~Remond-Tiedrez, L.M.~Smith, S.N.~Stechmann,
Beyond linear decomposition: a nonlinear eigenspace decomposition for a moist atmosphere with clouds.
arXiv:2405.11107.

\bibitem{Shi:15}
Y.~Shibata, 
On some free boundary problem of the Navier-Stokes equations in the maximal $\rL^p$-$\rL^q$-regularity class.
{\it J. Differential Equations}~{\bf 258} (2015), 4127--4155.

\bibitem{Ste:05}
B.~Stevens, 
Atmospheric moist convection.
{\it Annu. Rev. Earth Planet Sci.}~{\bf 33} (2005), 605--643.

\bibitem{Tri:78}
H.~Triebel,
{\it Interpolation Theory, Function Spaces, Differential Operators}. North-Holland, 1978.

\bibitem{Tri:10}
H.~Triebel,
{\it Theory of Function Spaces}. Modern Birkh\"auser Classics, Birkh\"auser/ Springer Basel AG, Basel, 2010.

\bibitem{ZSS:21}
Y.~Zhang, L.M.~Smith, S.N.~Stechmann, 
Effects of clouds and phase changes on fast-wave averaging: a numerical assessment.
{\it J. Fluid Mech.}~{\bf 920} (2021), Paper No.~A49.

\end{thebibliography}
\end{document}